\documentclass[12pt,a4paper]{amsart}
\usepackage[top=1.14in, bottom=1.14in, left=1.16in, right=1.16in]{geometry}
\usepackage[colorlinks=true, pdfstartview=FitV, linkcolor=blue, citecolor=blue, urlcolor=blue, breaklinks=true]{hyperref}

\usepackage[english]{babel}
\usepackage[T1]{fontenc}
\usepackage{latexsym}
\usepackage{amsmath}
\usepackage{amssymb, amscd, amsthm, mathrsfs}
\usepackage[all]{xy}
\usepackage[dvips]{graphicx}

\numberwithin{equation}{section}

\newtheorem{definition}{Definition}
\newtheorem{proposition}{Proposition}
\newtheorem{theorem}{Theorem}
\newtheorem{lemma}{Lemma}

\newtheorem{corollary}{Corollary}
\newtheorem{remark}{Remark}
\newtheorem{example}{Example}
\newtheorem{conjecture}{Conjecture}
\newtheorem*{thma}{Theorem A}
\newtheorem*{thmb}{Theorem B}
\newtheorem*{thmc}{Theorem C}

\newcommand{\g}{\mathfrak{g}}

\newcommand{\ra}{\rightarrow}
\newcommand{\ve}{\varepsilon}
\newcommand{\vp}{\varphi}
\newcommand{\ts}{\otimes}

\newcommand{\ff}{{\rm Fl}}

\newcommand{\Hom}{\operatorname*{Hom}}

\newcommand{\SL}{\operatorname*{SL}}
\newcommand{\GL}{\operatorname*{GL}}
\newcommand{\Gr}{\operatorname*{Gr}}
\newcommand{\lie}{\mathfrak}

\newcommand{\sspan}{\operatorname*{span}}
\newcommand{\id}{\operatorname*{id}}

\newcommand{\pr}{{\operatorname*{pr}}}
\newcommand{\bi}{\mathbf{i}}

\newcommand{\Leq}{{\leq_{\textrm{deg}}}}
\begin{document}

\title{Linear degenerations of flag varieties}
\author{G. Cerulli Irelli}
\address{Giovanni Cerulli Irelli:\newline
Dipartimento  di  Matematica  "G.  Castelnuovo",  Sapienza Universit\'a di Roma, Piazzale Aldo Moro 5, 00185, Roma, Italy}
\email{cerulli.math@googlemail.com} 
\author{X. Fang} 
\address{Xin Fang:\newline
University of Cologne, Mathematical Institute, Weyertal 86--90, 50931 Cologne, Germany}
\email{xinfang.math@gmail.com}
\author{E. Feigin} 
\address{Evgeny Feigin:\newline National Research University Higher School of Economics, Department of Mathematics, 
Vavilova str. 7, 117312, Moscow, Russia, {\it and } Tamm Theory Division, Lebedev Physics Institute}
\email{evgfeig@gmail.com}
\author{G. Fourier} 
\address{Ghislain Fourier:\newline
Leibniz University Hannover, Institute for Algebra, Number Theory and Discrete Mathematics, Welfengarten 1, 30167 Hannover, Germany}
\email{fourier@math.uni-hannover.de}
\author{M. Reineke}
\address{Markus Reineke:\newline
Ruhr-Universit\"at Bochum, Faculty of Mathematics, Universit\"atsstra{\ss}e 150, 44780 Bochum, Germany}
\email{markus.reineke@rub.de}

\begin{abstract}
Linear degenerate flag varieties are degenerations of flag varieties as quiver Grassmannians. For type $A$ flag varieties, we obtain characterizations of flatness, irreducibility and normality of these degenerations via rank tuples. Some of them are shown to be isomorphic to Schubert varieties and 
can be realized as highest weight orbits of partially degenerate Lie algebras, generalizing the corresponding results on degenerate flag varieties. To study normality, cell decompositions of quiver Grassmannians are constructed in a wider context of equioriented quivers of type $A$. 
\end{abstract}
\maketitle

\section{Introduction}
Let $B$ be a Borel subgroup in the group $\SL_{n+1}$. The flag variety $\SL_{n+1}/B$ has an explicit realization in linear algebra terms. 
Namely, let $V$ be an $n+1$-dimensional vector space. Then $\SL_{n+1}/B$ is isomorphic to the variety of collections $V_1,\dots,V_n$ of
subspaces of $V$, such that $V_i\subset V_{i+1}$ and $\dim V_i=i$. One can think of $V_i$ as sitting inside its own copy of $V$.
Let us denote the identity maps ${\rm id}:V\to V$ by $f_i$. Then a point in the flag variety is a collection of subspaces $V_i\subset V$ such that 
$\dim V_i=i$ and $f_iV_i\subset V_{i+1}$.  
This construction can be generalized in a very straightforward way: namely, we allow the $f_i$ to be arbitrary linear maps from $V$ to $V$.     
We denote the resulting variety by  ${\rm Fl}^{f_*}(V)$, where $f_*$ is the collection of maps $f_i$. 
The varieties ${\rm Fl}^{f_*}(V)$ can be naturally seen as degenerations of the classical flag variety $\SL_{n+1}/B$ (which corresponds to $f_i={\rm id}$);
we thus call ${\rm Fl}^{f_*}(V)$ the $f_*$-linear degenerate flag variety. 

Varying $f_*$, one can glue the varieties ${\rm Fl}^{f_*}(V)$ together into a universal object $Y$. By definition, there is a map $\pi$ from $Y$ 
to the parameter space $R$ of all possible $f_*$
(this is nothing but the product of $n-1$ copies of the space of linear endomorphisms from $V$ to $V$). We call $Y$ the universal linear degeneration  
of the flag variety. The main goal of the paper is to study the variety $Y$, the map $\pi:Y\to R$ and the fibers of the map $\pi$.

Our motivation comes from several different sources of representation theory and algebraic geometry. In \cite{Feigin2} the PBW degenerations $(G/B)^a$
of the classical flag varieties $G/B$ were constructed. The construction is of Lie-theoretic nature and works for arbitrary Lie groups. More precisely, one starts with 
an irreducible representation of $G$ and, using the PBW filtration on it, constructs the associated graded space. Then the degenerate flag variety
$(G/B)^a$ is the orbit closure of an abelian additive group acting on the projectivization of the PBW graded representation. Being applied to 
the case of $SL_{n+1}$, the construction produces the variety ${\rm Fl}^{f_*}(V)$ with all $f_i$ being corank one maps whose kernels are 
linear independent. It has been observed in \cite{CFR1} that methods of the theory of quiver Grassmannians can be used in order to study the 
properties of the PBW degenerations. Moreover, in \cite{CFR1} a family of well-behaved quiver Grassmannians was defined; these projective algebraic
varieties share many nice properties with the PBW degenerate flag varieties. Finally, in \cite{CL} the authors have identified the degenerate
flag varieties $(G/B)^a$ in types $A$ and $C$ with certain Schubert varieties (for larger rank groups). So for special values of $f_*$
the varieties  ${\rm Fl}^{f_*}(V)$ have nice geometric properties and many rich connections to structures of Lie theory and representation theory of quivers.
It is thus very reasonable to ask whether one can describe and study the  $f_*$-linear degenerate flag varieties for other $f_*$ and the global
(universal) degeneration $Y$.

We note that the parameter space $R$ is naturally
acted upon by the group $\GL(V)^n$. It is easy to see that the varieties ${\rm Fl}^{f_*}(V)$ and ${\rm Fl}^{g_*}(V)$ are isomorphic if $f_*$ and $g_*$ belong to
the same group orbit. The orbits of $\GL(V)^n$ are parametrized by tuples ${\bf r}=(r_{i,j})_{i<j}$ of ranks of the composite maps $f_{j-1}\circ\dots\circ f_i$.
We denote the orbit corresponding to the rank tuple ${\bf r}$ by  $\mathcal{O}_{{\bf r}}\subset R$. 
For example, the rank tuple ${\bf r}^0$ is defined by $r^0_{i,j}=n+1$; the corresponding $f_*$-degenerate flag varieties are isomorphic to the classical flag variety $\SL_{n+1}/B$.
We note that the orbit  $\mathcal{O}_{{\bf r}^0}$ degenerates to any other orbit $\mathcal{O}_{{\bf r}}$.   
The main result of the paper is the description of the following diagram of subsets of $R$:
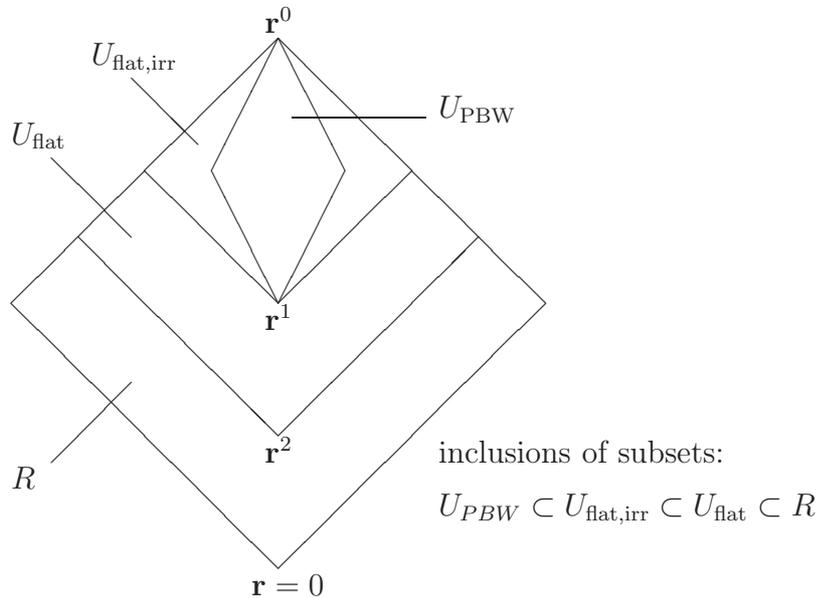
\bigskip
\begin{figure}[h]
\begin{picture}(200,200)
\put(100,0){$\line(1,1){100}$}
\put(100,200){$\line(1,-1){100}$}
\put(0,100){$\line(1,1){100}$}
\put(0,100){$\line(1,-1){100}$}
\put(25,125){$\line(1,-1){75}$}
\put(100,50){$\line(1,1){75}$}
\put(50,150){$\line(1,-1){50}$}
\put(100,100){$\line(1,1){50}$}
\put(100,200){$\line(1,-2){25}$}
\put(100,200){$\line(-1,-2){25}$}
\put(75,150){$\line(1,-2){25}$}
\put(125,150){$\line(-1,-2){25}$}
\put(0,30){$\scalebox{1}{$R$}$}
\put(15,40){$\line(1,1){30}$}
\put(0,160){$\scalebox{1}{$U_{\rm flat}$}$}
\put(15,155){$\line(1,-1){30}$}
\put(30,190){$\scalebox{1}{$U_{\rm flat, irr}$}$}
\put(45,185){$\line(1,-1){25}$}
\put(160,170){$\scalebox{1}{$U_{\rm PBW}$}$}
\put(155,170){$\line(-1,0){50}$}
\put(95,202){${\bf r}^0$}
\put(95,40){${\bf r}^2$}
\put(95,90){${\bf r}^1$}
\put(90,-10){${\bf r} = 0$}
\put(160,40){inclusions of subsets:} 
\put(160, 20){$U_{PBW}\subset U_{\rm flat,irr}\subset U_{\rm flat}\subset R$}
\end{picture}
\end{figure}

\bigskip
\par
The sets $U_{\bullet}$ are defined as follows:
\begin{itemize}
\item $U_{\rm flat}\subset R$ is the subset of $R$ over which $\pi$ is flat;
\item $U_{\rm flat, irr}\subset R$ is the subset of $R$ over which $\pi$ is flat with irreducible fibers;
\item $U_{\rm PBW}\subset R$ is the subset of $R$ where the kernels of the maps $f_i$ are at most one--dimensional and linearly independent. 
\end{itemize}
Each subset is a union of several $\GL(V)^n$ orbits; the properties of the $f_*$-degenerate flag variety strongly depend on the set $f_*$ belongs to.  
Let ${\bf r}^2$ be the rank tuple such that ${\bf r}^2_{i,j}=n-j+i$. 
Our first theorem gives a description of the largest set $U_{\rm flat}$. 

\begin{thma}
\begin{itemize}
\item[a)]
$U_{\rm flat}$ is the union of all orbits degenerating to $\mathcal{O}_{{\bf r}^2}$. 
\item[b)]
Over $U_{\rm flat}$ all fibers are reduced locally complete intersection varieties admitting a cellular decomposition.  
\item[c)]
the number of irreducible 
components of the fiber over a point of $\mathcal{O}_{{\bf r}^2}$ is equal to the $n$-th Catalan number. 
\end{itemize}
\end{thma}  

Let ${\bf r}^1$ be the rank sequence such that ${\bf r}^1_{i,j}=n+1-j+i$. 
The second theorem describes the flat locus with irreducible fibers.
\begin{thmb}
\begin{itemize}
\item[a)]
$U_{\rm flat, irr}$ is the union of all orbits degenerating to $\mathcal{O}_{{\bf r}^1}$. 
\item[b)] A $\pi$--fiber in $U_{\rm flat}$ is irreducible if and only if it is normal.
\item[c)]
The $\pi$-fiber over any point from $\mathcal{O}_{{\bf r}^1}$ is isomorphic to the PBW degenerate flag variety $(\SL_{n+1}/B)^a$. 
\end{itemize}
\end{thmb}  

Finally, the third theorem describes the locus of $f_*$, such that the structure of the varieties ${\rm Fl}^{f_*}(V)$ is very similar
to the PBW degenerate flag varieties.
\begin{thmc}
All $\pi$--fibers over points in $U_{\rm PBW}$ are PBW--type degenerations of the complete flag variety.
Each of them is acted upon by a unipotent group with an open dense orbit and it is naturally isomorphic  to a Schubert variety.
\end{thmc}  

The paper is organized as follows. 
\par
In Section \ref{setup}, we introduce the notion of linear degenerate flag varieties and explain the goal of this paper. Quiver Grassmannians are recalled in Section \ref{quiver}, and results on dimension estimation are obtained.
\par
Statements in Section \ref{setup} are translated into the language of quiver Grassmannians in Section \ref{Sec:Loci}; orbits and transversal slices in the flat irreducible locus are studied therein. We study the PBW locus in Section \ref{PBW}, where Theorem C is proved. 
In Section \ref{flat}, we prove Theorems A and B and study the desingularization of irreducible components of the fibers over $\mathcal{O}_{{\bf r}^2}$.
\bigskip
\\
\noindent
\textbf{Acknowledgements.} X.F. is supported by the Alexander von Humboldt Foundation. X.F would like to thank G.C-I. for invitation to Sapienza-Universit\`a di Roma where part of this work is carried out. G.C-I., G.F. and M.R. were partially supported by the DFG priority program 1388 ``Representation Theory``, in whose context this project has been initiated. E.F. was supported 
by the RSF--DFG grant 16-41-01013. 
G.C.I. was also supported by the italian FIRB program ``Perspectives in Lie Theory" RBFR12RA9W.

\section{Setup}\label{setup}
We fix the field of complex numbers $\mathbb{C}$ to be the base field.
\par 
Fix $n\geq 1$, and denote by $V$ an $(n+1)$-dimensional $\mathbb{C}$-vector space with basis $v_1,\ldots,v_{n+1}$. We consider sequences of linear maps 
$$
\xymatrix{
V\ar^-{f_1}[r]&V\ar^-{f_2}[r]&\ar[r]\cdots\ar^-{f_{n-1}}[r]&V,
}
$$
denoted as $f_*=(f_1,\ldots,f_{n-1})$.
These can be viewed as closed points of the variety $$R={\rm Hom}(V,V)^{n-1},$$
on which the group $G={\rm GL}(V)^n$ with elements $g_*=(g_1,\ldots,g_n)$  acts via base change:
$$g_*\cdot f_*=(g_2f_1g_1^{-1},g_3f_2g_2^{-1},\ldots,g_nf_{n-1}g_{n-1}^{-1}).$$
This action admits an open orbit $$U_{\rm iso}=G\cdot({\rm id}_V,\ldots,{\rm id}_V)$$ consisting of tuples of isomorphisms.

Let $U_*=(U_1,\ldots,U_n)$ be a tuple of subspaces in $V$ such that $\dim U_i=i$,  for all $i=1,\ldots,n$. Such tuples are viewed as the closed points of the product of Grassmannians
$$Z={\rm Gr}_1(V)\times\ldots\times{\rm Gr}_n(V),$$
which is a homogeneous space for the natural translation action of $G$ given by
$$g_*\cdot U_*=(g_1U_1,\ldots,g_nU_n).$$

We call $f_*$ and $U_*$ compatible if
$$f_i(U_i)\subset U_{i+1}\mbox{ for all }i=1,\ldots,n-1.$$
Let $Y$ be the variety of compatible pairs of sequences of maps and sequences of subspaces, that is,
$$Y=\{(f_*,U_*),\, f_i(U_i)\subset U_{i+1}\mbox{ for all }i=1,\ldots,n-1\}\subset R\times Z.$$
The actions of $G$ on $R$ and $Z$ induce an action of $G$ on $Y$. The projection $p:Y\rightarrow Z$ is $G$-equivariant, turning $Y$ into a homogeneous fibration over $Z$. The $p$-fiber over a tuple $(U_1,\ldots,U_n)$ can be identified, fixing complements $V_i$ to $U_i$ in $V$ for all $i$, with
$$\prod_{i<n}\left({\rm Hom}(U_i,U_{i+1})\oplus{\rm Hom}(V_i,V)\right),$$
thus $p$ is a homogeneous vector bundle over $Z$. In particular, the variety $Y$ is smooth and irreducible. To summarize the setup so far, we have the following diagram of $G$-equivariant varieties and maps:
$$
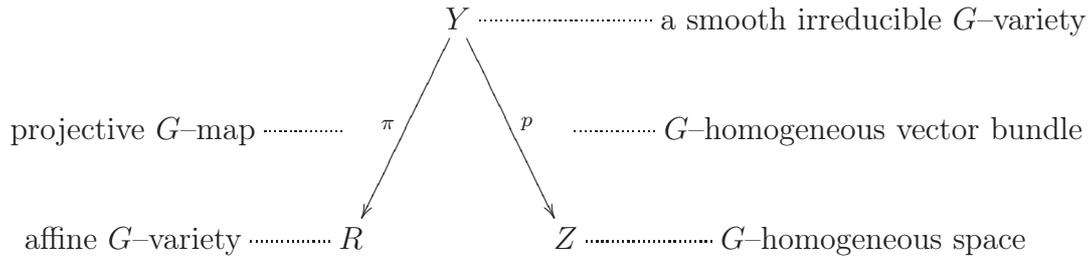
\xymatrix{
&&Y\ar_{\pi}[ddl]\ar^p[ddr]&&\ar@{..}[ll]\textrm{a smooth irreducible $G$--variety}\\
\textrm{projective $G$--map}\ar@{..}[r]&&&&\ar@{..}[l]\textrm{$G$--homogeneous vector bundle}\\
\textrm{affine $G$--variety}\ar@{..}[r]&R&&Z&\ar@{..}[l]\textrm{$G$--homogeneous space}
}
$$

The projection $\pi:Y\rightarrow R$ is projective, and the fiber over $({\rm id}_V,\ldots,{\rm id}_V)$ can be identified with the complete flag variety 
${\rm Fl}(V)$. Every fiber $\pi^{-1}(f_*)$ of $\pi$ can thus be viewed as a degenerate version of the complete flag variety.
\begin{definition} 
For $f_*\in R$, we call $${\rm Fl}^{f_*}(V)=\pi^{-1}(f_*)$$ the $f_*$-linear degenerate flag variety. We call the map $\pi:Y\rightarrow R$ the {\it universal linear degeneration} of ${\rm Fl}(V)$. We define $U_{\rm flat}\subset R$ as the subset of $R$ over which $\pi$ is flat, and $U_{\rm flat, irr}\subset R$ as the subset of $R$ over which $\pi$ is flat with irreducible fibers.
\end{definition}

By definition, we have
$$U_{\rm iso}\subset U_{\rm flat,irr}\subset U_{\rm flat}\subset R.$$

Our aim is to describe these loci and to study the geometry of the corresponding linear degenerate flag varieties.

\section{Methods from the representation theory of quivers}\label{quiver}

\subsection{Quiver representations}

For all basic definitions and facts on the representation theory of (Dynkin) quivers, we refer to \cite{ASS}.

Let $Q$ be a finite quiver with the set of vertices $Q_0$ and arrows written $\alpha:i\rightarrow j$ for $i,j\in Q_0$. We  assume that $Q$ is a Dynkin quiver, that is, 
its underlying unoriented graph $|Q|$ is a disjoint union of simply-laced Dynkin diagrams. 
\par
We consider (finite-dimensional) $\mathbb{C}$-representations of $Q$. Such a representation is given by a tuple $$V=((V_i)_{i\in Q_0},(f_\alpha)_{\alpha:i\rightarrow j}),$$
where $V_i$ is a finite-dimensional $\mathbb{C}$-vector space for every vertex $i$ of $Q$, and $f_\alpha:V_i\rightarrow V_j$ is a $\mathbb{C}$-linear map for every arrow $\alpha:i\rightarrow j$ in $Q$. A morphism between representations $V$ and $W=((W_i)_i,(g_\alpha)_\alpha)$ is a tuple of $\mathbb{C}$-linear maps $(\varphi_i:V_i\rightarrow W_i)_{i\in Q_0}$ such that $\varphi_jf_\alpha=g_\alpha\varphi_i$ for all $\alpha:i\rightarrow j$ in $Q$. Composition of morphisms is defined componentwise, resulting in a $\mathbb{C}$-linear category ${\rm rep}_\mathbb{C}Q$ of $Q$-representations of $Q$.
\par
This category is $\mathbb{C}$-linearly equivalent to the category $\bmod A$ of finite-dimensional left modules over the path algebra $A=\mathbb{C} Q$ of $Q$, in particular, it is a $\mathbb{C}$-linear abelian category.
\par
For a vertex $i\in Q_0$, we denote by $S_i$ the simple representation associated to $i$, namely, $(S_i)_i=\mathbb{C}$ and $(S_i)_j=0$ for all $j\not=i$, and all maps being identically zero; every simple representation is of this form. We let $P_i$ be a projective cover of $S_i$, and $I_i$ an injective hull of $S_i$. Considering $A$ as a left module over itself and using the above identification between representations of $Q$ and modules over $A$, we have $A=\bigoplus_{i\in Q_0}P_i$ and $A^*=\bigoplus_{i\in Q_0}I_i$, where $A^*$ denotes the $\mathbb{C}$-linear dual of $A$, viewed as a left module over $A$ with the aid of the right module structure of $A$ over itself.
\par
The Grothendieck group $K_0({\rm rep}_\mathbb{C}Q)$ is isomorphic to the free abelian group ${\mathbb Z}Q_0$ in $Q_0$ via the map attaching to the class of a representation $V$ its dimension vector 
${\rm\bf dim} V=(\dim V_i)_{i\in Q_0}\in{\mathbb Z}Q_0$. The category ${\rm rep}_\mathbb{C}Q$ is hereditary, that is, ${\rm Ext}^{\geq 2}(\_,\_)$ vanishes identically, and its homological Euler form
$$\dim {\rm Hom}(V,W)-\dim {\rm Ext}^1(V,W)=\langle{\rm\bf dim} V,{\rm\bf dim} W\rangle$$
is given by
$$\langle{\bf d},{\bf e}\rangle=\sum_{i\in Q_0}d_ie_i-\sum_{\alpha:i\rightarrow j}d_ie_j.$$

By Gabriel's theorem, the isomorphism classes $[U_\alpha]$ of indecomposable representations $U_\alpha$ of $Q$ correspond bijectively to the positive roots $\alpha$ of the root system 
$\Phi$ of type $|Q|$; more concretely, we realize $\Phi$ as the set of vectors $\alpha\in{\mathbb Z}Q_0$ such that $\langle\alpha,\alpha\rangle=1$; then there exists a unique (up to isomorphism) indecomposable representation $U_\alpha$ such that ${\rm\bf dim} U_\alpha=\alpha$ for every $\alpha\in\Phi^+=\Phi\cap{\mathbb N}Q_0$.

We make our discussion of the representation theory of a Dynkin quiver so far explicit in the case of the equioriented type $A_n$ quiver $Q$ given as
$$
\xymatrix{
1\ar[r]&2\ar[r]&\ar[r]\cdots\ar[r]&n
}
$$We identify ${\mathbb Z}Q_0$ with ${\mathbb Z}^n$, and the Euler form is then given by
$$\langle{\bf d},{\bf e}\rangle=\sum_{i=1}^nd_ie_i-\sum_{i=1}^{n-1}d_ie_{i+1}.$$
We denote the indecomposable representations by $U_{i,j}$ for $1\leq i\leq j\leq n$, where $U_{i,j}$ is given as
$$0\rightarrow\ldots\rightarrow 0\rightarrow \mathbb{C}\stackrel{{\rm id}}{\rightarrow}\ldots\stackrel{{\rm id}}{\rightarrow}\mathbb{C}\rightarrow 0\rightarrow\ldots\rightarrow 0,$$
supported on the vertices $i,\ldots,j$. In particular, we have $S_i=U_{i,i}$, $P_i=U_{i,n}$, $I_i=U_{1,i}$ for all $i$.

We have
$$\dim{\rm Hom}(U_{i,j},U_{k,l})=1\mbox{ if and only if }k\leq i\leq l\leq j$$
and zero otherwise, and we have
$$\dim{\rm Ext}^1(U_{k,l},U_{i,j})=1\mbox{ if and only if }k+1\leq i\leq l+1\leq j,$$
and zero otherwise, where the extension group, in case it is non-zero, is generated by the class of the exact sequence
$$0\rightarrow U_{i,j}\rightarrow U_{i,l}\oplus U_{k,j}\rightarrow U_{k,l}\rightarrow 0,$$ 
where we formally set $U_{i,j}=0$ if $i<1$ or $j>n$ or $j<i$.

\subsection{Varieties of representations of quivers}

Given a dimension vector ${\bf d}\in{\mathbb Z}Q_0$ and $\mathbb{C}$-vector spaces $V_i$ of dimension $d_i$ for $i\in Q_0$, let $R_{\bf d}$ be the affine $\mathbb{C}$-variety
$$R_{\bf d}=\bigoplus_{\alpha:i\rightarrow j}{\rm Hom}_\mathbb{C}(V_i,V_j),$$
on which the group
$$G_{\bf d}=\prod_{i\in Q_0}{\rm GL}(V_i)$$
acts via base change
$$(g_i)_i\cdot (f_\alpha)_\alpha=(g_jf_\alpha g_i^{-1})_{\alpha:i\rightarrow j}.$$
By definition, the $G_{\bf d}$-orbits $\mathcal{O}_M$ in $R_{\bf d}$ correspond bijectively to the isomorphism classes of representations $M$ of $Q$ of dimension vector ${\bf d}$. Note that, as a consequence of Gabriel's theorem, there are only finitely many $G_{\bf d}$-orbits in $R_{\bf d}$.

The orbit of $M$ degenerates to the orbit of $N$ if $N$ (or $\mathcal{O}_N$) is contained in the closure of $\mathcal{O}_M$. By \cite{Bo}, this holds if and only if $\dim{\rm Hom}(U,M)\leq\dim{\rm Hom}(U,N)$ for all indecomposable representations $U$ of $Q$.

\subsection{Dimension estimates for certain quiver Grassmannians}\label{Sec:QG}

Let $M$ be an arbitrary representation of the equioriented type $A_n$ quiver $Q$ of dimension vector ${\bf d}=(n+1,\ldots,n+1)$.
Let ${\bf e}$ be a dimension vector, ${\bf e}\le {\bf d}$ componentwise. As in Section \ref{setup}, let 
$Z_{\bf e}={\rm Gr}_1(V)\times\ldots\times{\rm Gr}_n(V)$ and let $Y_{\bf e}\subset R\times Z_{\bf e}$ be the variety of compatible pairs of
sequences $(f_*,U_*)$, $f_iU_i\subset U_{i+1}$. Then $Y_e={\rm Gr}_{\bf e}^Q({\bf d})$ is called the universal quiver Grassmannian. 
Let $\pi:Y_e\to R$ be the natural projection map. Then the quiver Grasmannian for a $Q$ representation $M\in R$ is defined as ${\rm Gr}_{\bf e}(M)=\pi^{-1}(M)$.
We would like to estimate the dimension of ${\rm Gr}_{\bf e}(M)$. A general representation $M^0$ of dimension vector ${\bf d}$ is isomorphic to 
$U_{1,n}^{\oplus(n+1)}$, thus all its arrows are represented by the identity maps. Since ${\rm Gr}_{\bf e}(M^0)$ is the $\SL_{n+1}$-flag variety, we know from \cite{CFR1} that every irreducible component of ${\rm Gr}_{\bf e}(M)$ has dimension at least $n(n+1)/2$. We would like to know when 
$\dim{\rm Gr}_{\bf e}(M)=n(n+1)/2$, and in case the equality holds, how many irreducible components (necessarily of this dimension) does the quiver Grassmannian have.

To this aim, we utilize a stratification of ${\rm Gr}_{\bf e}(M)$ introduced in \cite{CFR1}. Namely, for a representation $N$ of dimension vector ${\bf e}$, let $\mathcal{S}_{[N]}$ be the subset of ${\rm Gr}_{\bf e}(M)$ consisting of all sub-representations $U\subset M$ which are isomorphic to $N$. Then $\mathcal{S}_{[N]}$ is known to be an irreducible locally closed subset of ${\rm Gr}_{\bf e}(M)$ of dimension $\dim{\rm Hom}(N,M)-\dim{\rm End}(N)$. Since this gives a stratification of ${\rm Gr}_{\bf e}(M)$ into finitely many irreducible locally closed subsets, the irreducible components of ${\rm Gr}_{\bf e}(M)$ are necessarily of the form $\overline{\mathcal{S}_{[N]}}$ for certain $N$.

In the following, we decompose a representation $N$ as $N=N_P\oplus\overline{N}$, where $N_P$ is projective, and $\overline{N}$ has no projective direct summands. We decompose $M=P\oplus X$ where $P$ is projective. We first note a result which is very special to the linearly oriented type $A$ quiver $Q$ used here:

\begin{proposition}  Let $M$ and $N$ be as before and let $\mathbf{e}:=\mathbf{dim}\,N$ be a dimension vector. Then $N$ admits an embedding into $M$ if and only if
\begin{enumerate}
\item $\overline{N}$ embeds into $X$ and 
\item ${\bf e}-{\rm\bf dim}\,\overline{N}\leq{\rm\bf dim}\,P.$
\end{enumerate}
The isomorphism types of subrepresentations of $M$ of dimension vector ${\bf e}$ are para\-metrized by the isomorphism classes of representations $\overline{N}$ satisfying these two properties.
\end{proposition}

\begin{proof} This follows using the exact criterion given in \cite[Section 3]{MoeR} for existence of embeddings between representations of $Q$:

Suppose that $\overline{N}$ embeds into $X$ and that ${\bf e}-{\rm\bf dim}\overline{N}\leq{\rm\bf dim}P$. Then ${\rm\bf dim}N_P\leq {\rm\bf dim}P$, thus there exists an embedding of $N_P$ into $P$. This yields an embedding of $N$ into $M$. Conversely, suppose $N$ embeds into $M$. Then $\overline{N}$ embeds into $X$ since there are no non-zero maps from $\overline{N}$ to $P$. Now the special form of the inequalities in \cite[Section 3]{MoeR} characterizing embeddings also shows that $N_P$ embeds into $P$, which translates to ${\rm\bf dim}N_P\leq{\rm \bf dim}P$, yielding the second condition. Now given $\overline{N}$ satisfying both conditions, we define $N_P$ as the unique projective representation of dimension vector ${\bf e}-{\rm\bf dim}N_P$, which again embeds into $P$, thus determining the representation $N$.
\end{proof}

\begin{theorem}\label{tc} Let $M$ be a representation of $Q$ of dimension vector ${\bf d}$, written as $M=P\oplus X$, where $P$ is a projective representation. Let $\mathbf{e}:=\mathbf{dim}\,A=(1,2,\cdots, n)$.
\begin{enumerate}
\item The quiver Grassmannian ${\rm Gr}_{\bf e}(M)$ has dimension $n(n+1)/2$ if and only if, for all subrepresentations $\overline{N}$ of $X$ such that ${\bf e}-{\rm\bf dim}\overline{N}$ is the dimension vector of a projective representation embedding into $P$, we have
$$\dim{\rm End}(\overline{N})\geq \dim{\rm Hom}(\overline{N},X)-\dim{\rm Hom}(\overline{N},A^*).$$
\item In this case, the irreducible components of ${\rm Gr}_{\bf e}(M)$ are of the form $\overline{\mathcal{S}_{[N]}}$ for representations $N=N_P\oplus\overline{N}$ as above such that, in the previous inequality for $\overline{N}$, equality holds.
\end{enumerate}
\end{theorem}

\begin{proof}
Let $N$ be a subrepresentation of $M$ written as $N=N_P\oplus\overline{N}$ as above.

We use the shorthand notation $\dim{\rm Hom}(V,W)=[V,W]$ and the fact that there are no homomorphisms from representations without projective direct summands to projective representations. Then we can calculate:  
\begin{eqnarray*}
& &\dim\mathcal{S}_{[N]}-n(n+1)/2\\
&=&[N,M]-[N,N]-[A,A^*]\\
&=&[\overline{N},M]-[\overline{N},N]+[N_P,M]-[N_P,N]-[A,A^*]\\
&=&[\overline{N},X]-[\overline{N},\overline{N}]+\langle{\rm\bf dim} N_P,{\rm\bf dim}M-{\rm\bf dim}N\rangle-\langle{\rm\bf dim}A,{\rm\bf dim}A^*\rangle\\
&=&[\overline{N},X]-[\overline{N},\overline{N}]+\langle{\rm\bf dim} N_P,{\rm\bf dim}A^*\rangle-\langle{\rm\bf dim}A,{\rm\bf dim}A^*\rangle\\
&=&[\overline{N},X]-[\overline{N},\overline{N}]-\langle{\rm\bf dim} \overline{N},{\rm\bf dim}A^*\rangle\\
&=&[\overline{N},X]-[\overline{N},\overline{N}]-[\overline{N},A^*].
\end{eqnarray*}

All claims of the theorem follow.
\end{proof}

\subsection{Some local properties of schemes}

We collect some facts on local properties of schemes and morphisms which will be used in the following.

\begin{theorem}\label{ag} The following holds:
\begin{enumerate}
\item Let $f:X\rightarrow Y$ be a morphism of varieties, where $X$ is Cohen-Macaulay and $Y$ is regular. Then $f$ is flat if and only if its fibers are equidimensional.
\item Let $f:X\rightarrow Y$ be a flat proper morphism of varieties. Then the locus of all $y\in Y$ for which $f^{-1}(y)$ is reduced (resp.~irreducible, resp.~normal) is open in $Y$.
\item A scheme is reduced if it is generically reduced and Cohen-Macaulay.
\end{enumerate}
\end{theorem}

\begin{proof} The first statement is \cite[Theorem 23.1]{Mat}, the second is \cite[Theoreme 12.2.4]{EGAIV}, the third follows from reducedness being 
equivalent to generic reducedness plus Serre's condition $S_1$, and Cohen-Macaulay being equivalent to $S_k$ for all $k$.
\end{proof}

\section{Loci in \texorpdfstring{$R$}{the base space}}\label{Sec:Loci}

We translate the objects introduced in Section \ref{setup} to the language of quiver Grassmannians. Let $Q$ be the equioriented type $A_n$ quiver, 
$A$ be its path algebra, ${\bf d}=(n+1,\ldots,n+1)$ and ${\bf e}=(1,2,\ldots,n)$. Let $M\in R$ be a ${\bf d}$-dimensional $Q$-representation with the maps
$f_i:M_i\to M_{i+1}$. Then ${\rm Gr}_{\bf e}(M)$ is isomorphic to ${\rm Fl}^{f_*}(V)$ and the isomorphism is induced via the identification
of $R$ with the variety $R_{\bf d}(Q)$ of ${\bf d}$-dimensional representation of $Q$.
\par
The orbits of $G$ in $R$ are parametrized by rank tuples $${\bf r}=(r_{i,j})_{1\leq i<j\leq n},$$
where
$$r_{i,j}={\rm rank}(f_{j-1}\circ\ldots\circ f_i).$$
Denote by $\mathcal{O}_{\bf r}$ the set of all sequences of maps with the given ranks. This is non-empty if and only if a set of natural inequalities 
in the ranks is fulfilled, namely if
$$r_{i,j}+r_{i-1,j+1}\geq r_{i,j+1}+r_{i-1,j}$$
for all $1\leq i\leq j\leq n$, where we formally set $r_{i,j}=0$ if $i=0$ or $j=n+1$ and $r_{i,i}=n+1$. If non-empty, $\mathcal{O}_{\bf r}$ is a single $G$-orbit, and every orbit arises in this way. In particular, we have $U_{\rm iso}=O_{{\bf r}^0}$, where $r^0_{i,j}=n+1$ for all $i<j$.
\par
Moreover, it is known that $\mathcal{O}_{\bf r}$ degenerates to $\mathcal{O}_{{\bf r}'}$, that is, $\mathcal{O}_{{\bf r}'}$ is contained in the closure of $\mathcal{O}_{\bf r}$, if and only if $r_{i,j}\geq r_{i,j}'$ for all $i<j$.
\par
Denote by ${\rm Fl}^{\bf r}(V)$ the $\pi$-fiber over a point in $\mathcal{O}_{\bf r}$, which is well-defined up to isomorphism since $\pi$ is $G$-equivariant. We call ${\rm Fl}^{\bf r}(V)$ the ${\bf r}$-degenerate flag variety.

\begin{definition}\label{Def:R1R2}
We denote by ${\bf r}^1$ the rank tuple defined by ${\bf r}^1_{i,j}=n+1-j+i$ for all $i<j$, and by ${\bf r}^2$ the rank tuple defined by ${\bf r}^2_{i,j}=n-j+i$ for all $i<j$,
\end{definition}

\begin{theorem} \label{t2} We have the following description of the flat, respectively flat and irreducible, locus of $R$:
\begin{enumerate}
\item $U_{\rm flat}$ is the union of all orbits degenerating to $\mathcal{O}_{{\bf r}^2}$.
\item $U_{\rm flat, irr}$ is the union of all orbits degenerating to $\mathcal{O}_{{\bf r}^1}$. 
\end{enumerate}
\end{theorem}

The proof of Theorem~\ref{t2} will be given in Section~\ref{Sec:Proofs}.

\subsection{Complements of certain open loci in \texorpdfstring{$R$}{the base space}}

In light of the above interpretation of $R$ as $R_{\bf d}(Q)$, the $G$-orbits in $R$ are naturally parametrized by isomorphism classes of representations of $Q$ of dimension vector ${\bf d}$. 
By the Krull-Schmidt theorem, a $Q$-representation $M$ is, up to isomorphism, determined by the multiplicities of the $U_{i,j}$, that is,
$$M=\bigoplus_{i\leq j}U_{i,j}^{m_{i,j}}.$$
Then ${\rm\bf dim} M={\bf d}$ is equivalent to
$$\sum_{k\leq i\leq l}m_{k,l}=n+1\mbox{ for all }i.$$
We define
$$r_{i,j}(M)=\sum_{k\leq i\leq j\leq l}m_{k,l}$$
for $i\leq j$. Viewing $M$ as a tuple of maps $(f_1,\ldots,f_{n-1})$ as before, $r_{i,j}$ is thus the rank of $f_{j-1}\circ\ldots\circ f_i$ and, trivially, we have $r_{i,i}=n+1$. We can recover $m_{i,j}$ from $(r_{k,l})_{k,l}$ via
$$m_{i,j}=r_{i,j}-r_{i,j+1}-r_{i-1,j}+r_{i-1,j+1},$$
which explains the natural rank inequalities above. More generally, we easily derive the inequality
\begin{equation}\label{recin}r_{i,l}+r_{j,k}\geq r_{i,k}+r_{j,l}\end{equation}
for all four-tuples $i<j\leq k<l$.

We introduce some special representations: for a tuple ${\bf a}=(a_1,\ldots,a_{n-1})$ of non-negative integers $a_i$ such that $\sum_{i<n}a_i\leq n+1$, we define $M({\bf a})$ by the multiplicities:
$$m_{1,n}=n+1-\sum_ia_i,\; m_{1,i}=a_i\mbox{ for }i<n,\; m_{i,n}=a_{i-1}\mbox{ for }i>1,$$
and $ m_{j,k}=0\mbox{ for all other }j<k.$
In particular, we define
$$M^0=M(0,\ldots,0),\;
M^1=M(1,\ldots,1).$$

We also define $M^2$ by the multiplicities $$m_{1,1}=m_{n,n}=2,\; m_{1,i}=1\mbox{ for all }i>1,\; m_{i,n}=1\mbox{ for all }i<n,$$
$$ m_{i,i}=1\mbox{ for all }1<i<n,$$
and $m_{j,k}=0$ for all other $j<k$.

A direct calculation then shows that
$${\bf r}(M^0)={\bf r}^0,\;
{\bf r}(M^1)={\bf r}^1,\;
{\bf r}(M^2)={\bf r}^2.$$

In more invariant terms, we can write $M^1=A\oplus A^*$, where $A$ is the path algebra viewed as a (bi-)module over itself, and $A^*$ is the linear dual of $A$. There exists a short exact sequence
$$0\rightarrow A\rightarrow M^0\rightarrow A^*\rightarrow 0.$$
We have canonical maps
$$A\rightarrow A/{\rm rad}(A)=:S\simeq{\rm soc}(A^*)\rightarrow A^*,$$
and $M^2$ can be written as
\begin{equation}\label{Eq:DefM2}
M^2\simeq A\oplus S\oplus (A^*/S)\simeq {\rm rad}(A)\oplus S\oplus A^*.
\end{equation}

Now we turn to degenerations of representations. Again we write $M\leq N$ if the closure of the $G_{\bf d}$-orbit of $M$ contains $N$; the numerical characterization \cite{Bo} of degenerations mentioned above then reads
$$M\leq N\mbox{ if and only if } r_{i,j}(M)\geq r_{i,j}(N)\mbox{ for all }i<j.$$
The representation $M^0=U_{1,n}^{n+1}$ is generic in the sense that $M^0\leq M$ for all $M$ in $R$.
It is proven in \cite{CFR1} that a representation $M$ degenerates to $M^1$ if and only if it fits into a short exact sequence $0\rightarrow A\rightarrow M\rightarrow A^*\rightarrow 0$.

We are now interested in the complement of the locus of representations degenerating into $M^1$ resp. $M^2$. For this, we introduce the following tuples:
\begin{itemize}
\item for $1\leq i<n$, define $${\bf a}^i=(0,\ldots,0,2,0,\ldots,0),$$ with the $2$ placed at the $i$-th entry;
\item  for $1\leq i\leq j<n$, define $${\bf a}^{i,j}=(0,\ldots,0,2,1,\ldots,1,2,0,\ldots,0),$$
with the $2$'s placed at the $i$-th and $j$-th entry, except in the case $j=i$, where we define   
$${\bf a}^{i,i}=(0,\ldots,0,3,0,\ldots,0),$$ with the $3$ placed at the $i$-th entry.
\end{itemize}

Now we can formulate:

\begin{theorem}\label{t5} Let $M$ be a representation in $R$.
\begin{enumerate}
\item If $M$ degenerates to $M^2$ but not to $M^1$, then $M$ is a degeneration of $M({\bf a}^i)$ for some $i$.
\item If $M$ does not degenerate to $M^2$, then $M$ is a degeneration of $M({\bf a}^{i,j})$ for some $i\leq j$.
\end{enumerate}
\end{theorem}

\begin{proof} To prove the first part, let $M$ degenerate to $M^2$ but not to $M^1$ and consider the corresponding rank system ${\bf r}={\bf r}(M)$. 
Degeneration of $M$ to $M^2$ is equivalent to ${\bf r}\geq{\bf r}^2$ componentwise, thus $r_{i,j}\geq n-j+i$ for all $i<j$. Non-degeneration of $M$ to $M^1$ is equivalent to ${\bf r}\not\geq{\bf r}^1$, thus there exists a pair $i<j$ such that $r_{i,j}<n-j+i+1$, which implies $r_{i,j}=n-j+i$. We claim that this equality already holds for a pair $i<j$ such that $j=i+1$. Suppose, to the contrary, that $r_{i,j}=n-j+i$ for some pair $i<j$ such that $j-i\geq 2$, and that $r_{k,l}\geq n-l+k+1$ for all $k<l$ such that $l-k<j-i$. In particular, we can choose an index $k$ such that $i<k<j$, and the previous estimate holds for $r_{i,k}$ and $r_{k,j}$. But then, the inequality (\ref{recin}), applied to the quadruple $i<k=k<j$ yields
$$2n+1-j+i=r_{i,j}+r_{k,k}\geq r_{i,k}+r_{k,j}=2n+2-j+i,$$
a contradiction. We thus find an index $i$ such that $r_{i,i+1}=n-1$, and thus $r_{k,l}\leq n-1$ for all $k\leq i<i+1\leq l$ trivially. On the other hand, it is easy to compute the rank tuple of $M({\bf a}^i)$ as
$$r_{j,k}(M({\bf a}^i))=n-1\mbox{ for }j\leq i<k,$$
and $r_{j,k}(M({\bf a}^i))=n+1$ otherwise. This proves that ${\bf r}\leq{\bf r}(M({\bf a}^i))$ as claimed.
\par
Now suppose that $M$ does not degenerate to $M^2$, and again consider the rank system ${\bf r}={\bf r}(M)\not\geq {\bf r}^2$. We thus find a pair $i<j$  such that $$r_{i,j}\leq n-j+i-1.$$
We assume this pair to be chosen such that $j-i$ is minimal with this property; thus
$$r_{k,l}\geq n-l+k\mbox{ for all }k<l\mbox{ such that }l-k<j-i.$$ For every $i<k<j$, application of the inequality (\ref{recin}) to the quadruple $i<k=k<j$ yields
$$2n-j+i=(n-j+i-1)+(n+1)\geq r_{i,j}+r_{k,k}\geq$$
$$\geq r_{i,k}+r_{k,j}\geq (n-k+i)+(n-j+k)=2n-j+i,$$
from which we conclude
$$r_{i,k}=n-k+i,\, r_{k,j}=n-k+j\mbox{ for all }i<k<j$$
and
$$r_{i,j}=n-j+i-1.$$
Now we claim that
$$r_{k,l}=n-l+k+1\mbox{ for all }i<k<l<j.$$
This condition is empty if $j-i=1$, thus we can assume $j-i\geq 2$.
We prove this by induction over $k$, starting with $k=i+1$. For every $i+1<l<j$, application of (\ref{recin}) to $i<l-1<l<l$ yields
$$r_{i+1,l-1}=r_{i+1,l-1}+r_{i,l}-r_{i,l-1}+1\geq r_{i+1,l}+1.$$
This, together with (\ref{recin}) for $i<i+1\leq j-1<j$, yields the estimate
$$n+1=r_{i+1,i+1}\geq r_{i+1,i+2}+1\geq r_{i+1,i+3}+2\geq\ldots$$
$$\ldots\geq r_{i+1,j-1}+(j-i-2)\geq r_{i+1,j}+r_{i,j-1}-r_{i,j}+(j-i-2)=n+1,$$
thus equality everywhere. Now assume that $k>i+1$, and that the claim holds for all relevant $r_{k-1,l}$. Similarly to the previous argument, we arrive at an estimate
$$n+1=r_{k,k}\geq r_{k,k+1}+1\geq r_{k,k+2}\geq\ldots$$
$$\ldots\geq r_{k,j-1}+j-k-1\geq r_{k,j}+r_{k-1,j-1}-r_{k-1,j}+j-k-1=n+1,$$
and this again yields equality everywhere. This proves the claim.

Finally, we have the trivial estimates
\begin{itemize}
\item $r_{k,l}\leq r_{i,j}=n-j+i-1$ if $k\leq i\leq j\leq l$,
\item $r_{k,l}\leq r_{i,l}=n-l+i$ if $k<i<l<j$,
\item $r_{k,l}\leq r_{k,j}=n-j+k$ if $i<k<j<l$, and trivially
\item $r_{k,l}\leq n+1$ otherwise, that is, if $k<l\leq i<j$ or $i<j\leq k<l$.
\end{itemize}
A long but elementary calculation of ${\bf r}(M({\bf a}^{i,j}))$ shows that all these estimates together prove that
$${\bf r}\leq{\bf r}(M({\bf a}^{i,j})).$$ The theorem is proved.
\end{proof}

\subsection{Orbits in the flat irreducible locus}

We introduce a combinatorial object, generalizing the rhyme schemes of \cite{Rior}, to parametrize the orbits in the flat irreducible locus.

\begin{definition} A {\it broken rhyme scheme} of length $n-1$ is a sequence $(b_1,\ldots,b_{n-1})$ of non-negative integers such that $b_1\in\{0,1\}$ and $b_{i+1}\leq\max(b_1,\ldots,b_i)+1$ for all $i\leq n-2$. It is called {\it regular} if $b_i\not=b_j$ whenever $i\not=j$ and $b_i,b_j\not=0$.
\end{definition}

For example, the broken rhyme schemes of length $3$ (the regular ones being underlined) are:
$$\underline{000},\,\underline{001},\, \underline{010},\, 011,\, \underline{012},$$
$$\underline{100},\, 101,\, \underline{102},\, 110,\, 111,$$
$$112,\, \underline{120},\, 121,\, 122,\, \underline{123}.$$

\begin{proposition}\label{proprhyme} The $G$-orbits 
degenerating to $\mathcal{O}_{{\bf r}^1}$ are parametrized by broken rhyme schemes of length $n-1$. More precisely, to a broken rhyme scheme 
$(b_1,\ldots,b_{n-1})$ we associate the orbit of the sequence $(f_1,\ldots,f_{n-1})$, where $f_i={\rm id}_V$ if $b_i=0$, and 
$f_i={\rm pr}_{b_i}$ if $b_i\not=0$; here ${\pr}_k$ denotes the linear map given by projection along the $k$-th basis vector $v_k$ of $V$.
\end{proposition}

\begin{proof} An orbit degenerating to $\mathcal{O}_{{\bf r}^1}$ is uniquely determined by its rank tuple 
$(r_{i,j})_{i,j}$ satisfying $r_{i,j}\geq n+1-j+i$ for all 
$i$ and $j$. These conditions are fulfilled if and only if $r_{i,i+1}\geq n$ for all $i$, in other words if and only if 
$r_{i,i+1}\in\{n,n+1\}$ for all $i$. Thus, using the base change action, we find a point in this orbit given by linear maps $f_i={\rm id}_V$  
if $r_{i,i+1}=n+1$ (in which case we formally define $b_i=0$), and $f_i={\rm pr}_{b_i}$ for some $b_i\in\{1,\ldots,n+1\}$ if $r_{i,i+1}=n$. 
Using the natural $S_{n+1}$-action on each $V$, the resulting sequence of integers $(b_1,\ldots,b_{n-1})$ can be transformed into a broken rhyme 
scheme in a unique way.
\end{proof}

\begin{remark}
According to Theorem \ref{t2}, broken rhyme schemes parametrize the $G$-orbits in $U_{\rm flat,irr}$. 
\end{remark}

\begin{definition}\label{defpbw} 
We define the PBW locus $U_{\rm PBW}\subset U_{\rm flat,irr}$ as the union of the orbits corresponding to regular broken rhyme schemes.
\end{definition}

\begin{remark} The parametrization of orbits in the flat locus $U_{\rm flat}$ is less explicit; we mention without proof the following combinatorial description:

Consider the set $P$ of sequences $(I_1,\ldots,I_{n-1})$ of subsets of $\{1,\ldots,n+1\}$ with the following properties:
\begin{enumerate}
\item $|I_i|\leq 2$ for all $i$,
\item $|I_i\cup I_{i+1}|\leq 3$ for all $i$.
\end{enumerate}
The symmetric group $\mathfrak{S}_{n+1}$ acts on $P$ by permutation in each $I_i$. Then the $G$-orbits in $U_{\rm flat}$ are parametrized by $P/\mathfrak{S}_{n+1}$. Namely, to a sequence in $P$, we associate the sequence of linear maps $({\rm pr}_{I_1},\ldots,{\rm pr}_{I_{n-1}})$, where ${\rm pr}_I$ denotes projection along all basis vectors $v_i$ such that $i\in I$.
\end{remark}

\subsection{Transversal slice}\label{Ts}

We are interested in constructing transversal slices, that is, an affine subspace $T_{PBW}$ (resp.~$T$)  
of $R$ which is contained in $U_{\rm PBW}$ 
(resp.~$U_{\rm flat,irr}$), meets every $G$-orbit in $U_{\rm PBW}$ (resp.~$U_{\rm flat,irr}$), and intersects the minimal orbit $\mathcal{O}_{{\bf r}^1}$ 
in a single point. The construction is elementary for $T_{\rm PBW}$. Since the slice $T$ will not be needed in the rest of the paper, we 
just state the result -- the method for its construction is contained in \cite[Theorem 6.2]{Bo}.

As before, let $v_1,\dots,v_{n+1}$ be a basis of the space $V$.  
\begin{definition} Define $T\subset R$ as the subset of all tuples of linear maps $(f_1,\ldots,f_{n-1})$ such that
$$(f_i)(v_q)=\left\{\begin{array}{ccc}v_p&,& p=q\not=i+1,\\ \lambda_{p-1,q-1}v_p&,& 2\leq p\leq i+1\leq q\leq n,\\
0&,& \mbox{ otherwise }\end{array}\right.$$
for certain $(\lambda_{i,j})_{1\leq i\leq j\leq n-1}$.

Define $T_{\rm PBW}$ as the subspace of $T$ for which all $\lambda_{i,j}$ for $i<j$ are zero.
\end{definition}

For example, for $n=4$, the matrices representing the triples $(f_1,f_2,f_3)$ have the following form:

$$\left(\begin{array}{rrrrr}1&0&0&0&0\\ 0&\lambda_{11}&\lambda_{12}&\lambda_{13}&0\\ 0&0&1&0&0\\ 0&0&0&1&0\\ 0&0&0&0&1\end{array}\right)\, 
\left(\begin{array}{rrrrr}1&0&0&0&0\\ 0&1&\lambda_{12}&\lambda_{13}&0\\ 0&0&\lambda_{22}&\lambda_{23}&0\\ 0&0&0&1&0\\ 0&0&0&0&1\end{array}\right)\, 
\left(\begin{array}{rrrrr}1&0&0&0&0\\ 0&1&0&\lambda_{13}&0\\ 0&0&1&\lambda_{23}&0\\ 0&0&0&\lambda_{33}&0\\ 0&0&0&0&1\end{array}\right)$$

\begin{proposition} $T$ is a transversal slice in $U_{\rm flat,irr}$ in the above sense, and $T_{\rm PBW}$ is a transversal slice in $U_{\rm PBW}$.
\end{proposition}

\begin{proof} The second claim 
follows immediately from the parametrization of orbits in $U_{\rm PBW}$ 
as stated in Definition \ref{defpbw}, Proposition \ref{proprhyme}. To prove that $T$ is a transversal slice in $U_{\rm flat,irr}$, 
one applies the construction of \cite[Theorem 6.2]{Bo} to the representation $A\oplus A^*$.
\end{proof}

The utility of this transversal slice is that we can localize the universal degenerate flag variety, that is, we can consider the restriction of $\pi:\pi^{-1}(U_{\rm PBW})\rightarrow U_{\rm PBW}$ to $\pi:\pi^{-1}(T)\rightarrow T$ without losing any information; the base is now a much smaller affine space with an obvious stratification into strata over which $\pi$ is locally trivial. 

\section{Geometry of linear degenerations - the PBW locus}\label{PBW}

In this section we study  the geometry of linear degenerations in the PBW locus.
We prove that any degeneration from the PBW locus is isomorphic to a Schubert variety; we realize each such 
Schubert variety as the closure of a highest weight orbit.

\begin{theorem} 
All linear degenerations of flag varieties in the PBW locus are orbit closures of highest weight line in the projectivized PBW degenerations 
of irreducible representations. Moreover, they are isomorphic to Schubert varieties.
\end{theorem}

The goal of this section is to make this theorem explicit and to provide a proof.

\subsection{Projection sequences}
We start with another parametrization of the regular broken rhyme scheme via projection sequences. Let 
$$\mathcal{D}=\{\mathbf{i}=(i_1,i_2,\cdots,i_k)\in\mathbb{N}^k\,|\ 1\leq i_1<i_2<\cdots<i_k\leq n-1\}$$
be the set of sequences of numbers between $1$ and $n-1$ (the empty sequence $\emptyset$ is included).

\begin{lemma}
There exists a bijection between the set of regular broken rhyme schemes and $\mathcal{D}$.
\end{lemma}

\begin{proof}
For $\bi=(i_1,\cdots,i_k)\in\mathcal{D}$, we define a sequence $\mathbf{b}=(b_1,\cdots,b_{n-1})$ by:
\[ b_{i}= \begin{cases}
       s & \text{if } i=i_s\ \text{for some }s=1,\cdots,k;\\
       0 & \text{ otherwise}.
     \end{cases} \]
Then $\mathbf{b}$ is a regular broken rhyme scheme and it is clear that the above map is a bijection.
\end{proof}

For $\mathbf{i}=(i_1,i_2,\cdots,i_k)\in\mathcal{D}$ we denote the projection sequence $\text{pr}_{\bi}$ by:
$$\text{pr}_{\bi}=(\underbrace{\id,\cdots,\id}_{i_1-1\ \text{copies}},\text{pr}_{i_1+1},\underbrace{\id,\cdots,\id}_{i_2-i_1-1\ \text{copies}},\text{pr}_{i_2+1},\cdots,\text{pr}_{i_k+1},\underbrace{\id,\cdots,\id}_{n-1-i_k\ \text{copies}}).$$
We write $\text{pr}_{\bi}=(f_1,f_2,\cdots,f_{n-1})$ where $f_i$ is either the projection along a line or identity.
Then for any sequence $\bi\in\mathcal{D}$, the corresponding linear degenerate flag variety $\ff_{n+1}^{\,\bi}$ is the $f_*$-linear
degenerate flag variety ${\rm Fl}^{\text{pr}_{\bi}}(V)$.
\par
Recall that $A=\mathbb{C}Q$ is the path algebra of the quiver $Q$. We consider the quiver Grassmannian ${\Gr}_{\mathbf{e}}(M^\bi)$, where 
$\mathbf{e}={\bf dim}(A)=(1,2,\cdots,n)$ and
$$M^\bi=P_1^{\oplus n+1-k}\oplus \left(\bigoplus_{m=1}^k I_{i_m}\oplus P_{i_m+1}\right).$$

The following proposition holds by rephrasing the definition (see for example \cite[Proposition 2.7]{CFR1}).

\begin{proposition}
We have an isomorphism of projective varieties $$\ff_{n+1}^{\,\bi}\stackrel{\sim}{\longrightarrow}{\Gr}_{\mathbf{e}}(M^\bi).$$
\end{proposition}

\begin{example}
\begin{enumerate}
\item For $\bi=\emptyset\in\mathcal{D}$, $\ff_{n+1}^{\,\bi}\stackrel{\sim}{\longrightarrow}\SL_{n+1}/B$ is the complete flag variety.
\item For $\bi=(1,2,\cdots,n-1)\in\mathcal{D}$, $\ff_{n+1}^{\,\bi}\stackrel{\sim}{\longrightarrow}(\SL_{n+1}/B)^a$ is the degenerate flag variety \cite{Feigin2}.
\end{enumerate}
\end{example}

For $\bi=(i_1,i_2,\cdots,i_k)\in\mathcal{D}$, we denote $d(\bi):=k$. 
\begin{remark}\label{Rem:Schubert}
As shown in \cite{CFRSchubert}, every quiver Grassmannian associated with a representation M of the quiver Q, which is equioriented of type A, can be naturally embedded into a flag manifold. As shown in loc. cit. the image of such an embedding is stable under the action of a Borel subgroup if and only if M is a catenoid. 
In this case the irreducible components are Schubert varieties. A quiver Grassmannian associated with a catenoid is called a Schubert quiver Grassmannian. 
Since $M^{\bi}$ is the direct sum of a projective and an injective Q--representations, it is a catenoid. Since the quiver Grassmannian $\Gr_\mathbf{e}(M^\bi)$ is irreducible, it is a Schubert variety. In the following section, we will describe these Schubert varieties explicitly. 
\end{remark}

\subsection{Realization as Schubert varieties}
We fix $\bi=(i_1,i_2,\cdots,i_k)\in\mathcal{D}$ as before. Define $\mathbf{h}_{\bi}=(h_1,h_2,\cdots,h_n)$ by: 
$h_1=0$ and for any $s=2,3,\cdots,n$, 
$$h_s=\#\{t\,|\,1\leq t\leq k\text{ and }i_t <s\}.$$
We consider $\mathfrak{sl}_{n+1+d(\bi)}$ with Weyl group $W_{\bi}$ generated by the reflections $s_i$ 
with respect to the simple roots $\alpha_i$ of $\mathfrak{sl}_{n+1+d(\bi)}$.
We define $w_{\bi}=w_nw_{n-1}\cdots w_1\in W_{\bi}$ as follows:
$$w_k=s_{h_k+1}s_{h_k+2}\cdots s_{h_k+k}.$$
We denote $\ell_j := h_j + j$. 
\par
Let $\mathfrak{h}_{\bi}$ be the Cartan subalgebra of $\mathfrak{sl}_{n+1+d(\bi)}$ consisting of diagonal matrices and let  $\mathfrak{h}:=\mathfrak{h}_{\emptyset}$ be the Cartan subalgebra of $\mathfrak{sl}_{n+1}$. We define a map $\Psi^{\bi}:\mathfrak{h}^*\ra \mathfrak{h}_{\bi}^*$ by:
\begin{equation}
{\Psi^{\bi}}(\varpi_j)=\varpi_{\ell_j}.
\end{equation}
Let $\mathcal{P}^+$ ($\mathcal{P}_{\,\bi}^+$) be the set of dominant integral weights for $\mathfrak{sl}_{n+1}$ ($\mathfrak{sl}_{n+1+d(\bi)}$). 
For any $\lambda\in\mathcal{P}^+$, we define 
$$\lambda_{\bi}:={\Psi^{\bi}}(\lambda)\in \mathcal{P}_{\,\bi}^+\subset \mathfrak{h}_{\bi}^*.$$
Let $\rho=\varpi_1+\varpi_2+\cdots+\varpi_n$ and $\rho_{\bi}:=\Psi^\bi(\rho)\in\mathfrak{h}_{\bi}^*$.

We let $X_{w_\bi}$ denote the Schubert variety in ${\SL}_{n+1+d(\bi)}/P_{\rho_{\bi}}$ associated to $w_{\bi}$ where $P_{\rho_{\bi}}$ is the parabolic subgroup of ${\SL}_{n+1+d(\bi)}$ 
stabilizing the weight $\rho_{\bi}$.

\begin{theorem}\label{Thm:flagSchubert}
We have an isomorphism of projective varieties
$$\ff_{n+1}^\bi\stackrel{\sim}{\longrightarrow} X_{w_\bi}.$$
\end{theorem}

Before giving the proof of the theorem, we examine it in several known examples.

\begin{example}
\begin{enumerate}
\item For $\bi=\emptyset\in\mathcal{D}$, 
$$w_{\bi}=s_1s_2\cdots s_ns_1s_2\cdots s_{n-1}\cdots s_1s_2s_1=w_0$$
is the longest element in the Weyl group of $\mathfrak{sl}_{n+1}$. In this case, $X_{w_\bi}\subset {\SL}_{n+1}/\text{P}_{\rho}={\SL}_{n+1}/\text{B}$ is the complete flag variety.
\item For $\bi=(1,2,\cdots,n-1)\in\mathcal{D}$, 
$$w_\bi=(s_{n}s_{n+1}\cdots s_{2n-1})(s_{n-1}s_{n}\cdots s_{2n-3})\cdots (s_3s_4s_5)(s_2s_3)s_1,$$ 
and $X_{w_\bi}\subset {\SL}_{2n}/\text{P}_{\varpi_1+\varpi_3+\cdots+\varpi_{2n-1}}$ is isomorphic to the degenerate flag variety, as shown in \cite{CL}.
\end{enumerate}
\end{example}

The following example illustrates the above construction.

\begin{example}
Let $\g=\mathfrak{sl}_5$ be the simple Lie algebra of type $A_4$. In the following table we list all PBW linear degenerations of the complete flag variety of $\g$.
{\begin{tiny}
\begin{center}
    \begin{tabular}{ | c | c | c | c | c |}
    \hline
     $\bi\in\mathcal{D}$ & Projection & $w_\bi$ & $\mathbf{h}_\bi$ & $M$ \\ \hline
   $\emptyset$ & $(\id,\id,\id)$ & $s_1s_2s_3s_4s_1s_2s_3s_1s_2s_1$ & $(0,0,0,0)$ & $P_1^{\oplus 5}$  \\ \hline
  $\{1\}$ & $(\pr_2,\id,\id)$ & $s_2s_3s_4s_5s_2s_3s_4s_2s_3s_1$ & $(0,1,1,1)$ & $P_1^{\oplus 4}\oplus I_1\oplus P_2$  \\ \hline
  $\{2\}$ & $(\id,\pr_3,\id)$ & $s_2s_3s_4s_5s_2s_3s_4s_1s_2s_1$ & $(0,0,1,1)$ & $P_1^{\oplus 4}\oplus I_2\oplus P_3$  \\ \hline
  $\{3\}$ & $(\id,\id,\pr_4)$ & $s_2s_3s_4s_5s_1s_2s_3s_1s_2s_1$ & $(0,0,0,1)$ & $P_1^{\oplus 4}\oplus I_3\oplus P_4$  \\ \hline
  $\{1,2\}$ & $(\pr_2,\pr_3,\id)$ & $s_3s_4s_5s_6s_3s_4s_5s_2s_3s_1$ & $(0,1,2,2)$ & $P_1^{\oplus 3}\oplus I_1\oplus P_2\oplus I_2\oplus P_3$ \\ \hline
  $\{1,3\}$ & $(\pr_2,\id,\pr_4)$ &  $s_3s_4s_5s_6s_2s_3s_4s_2s_3s_1$ & $(0,1,1,2)$ & $P_1^{\oplus 3}\oplus I_1\oplus P_2\oplus I_3\oplus P_4$ \\ \hline
  $\{2,3\}$ & $(\id,\pr_3,\pr_4)$ & $s_3s_4s_5s_6s_3s_4s_5s_1s_2s_1$ & $(0,0,1,2)$ & $P_1^{\oplus 3}\oplus I_2\oplus P_3\oplus I_3\oplus P_4$ \\ \hline
   $\{1,2,3\}$ & $(\pr_2,\pr_3,\pr_4)$ & $s_4s_5s_6s_7s_3s_4s_5s_2s_3s_1$ & $(0,1,2,3)$ & $P_1^{\oplus 2}\oplus I_1\oplus P_2\oplus I_2\oplus P_3\oplus I_3\oplus P_4$ \\ \hline
    \end{tabular}
\end{center}
\end{tiny}}
\end{example}

For the proof of Theorem \ref{Thm:flagSchubert}, we will need the following general result.
\par
Let $\ff_{(\ell_1,\dots,\ell_n)}$ be the partial flag variety for ${\SL}_{n+1+d(\bi)}$ consisting of collections of subspaces 
$(U_i)_{i=1}^n$ of dimensions $\ell_1,\dots,\ell_n$.
\begin{proposition}
Let $w\in W_{\bi}$ be an element satisfying the following condition for all $j=1,\dots, n$:
\[
w(1,\dots,\ell_j)=\{1,\dots,\ell_j-j\}\cup \{n+1+\ell_j-j,n+1+\ell_j-(j-1),\dots, n+1+\ell_j-1\}.
\]
Then ${\Gr}_{\mathbf{e}}(M^\bi)$ is isomorphic to the Schubert variety $X_w$ attached to $w$
in the partial flag variety $\ff_{(\ell_1,\dots,\ell_n)}$.
\end{proposition}
\begin{proof}
This is a slight generalization of \cite[proof of Theorem~1.2]{CL} and a particular case of \cite[Theorem~3.4 and Proposition~3.8]{CFRSchubert}. We briefly recall the proof, for convenience of the reader. 
Let $v_1,\dots,v_{n+1+d(\bi)}$ be the standard basis of $\mathbb{C}^{n+1+d(\bi)}$.
Assume that the conditions on $w$ are fulfilled. Then the variety $X_w$ consists of collections of subspaces
\[
U_1\subset U_2\subset\dots\subset U_n\subset \mathbb{C}^{n+1+d(\bi)} 
\]
such that $\dim U_j=\ell_j$ and $\mathrm{span}_{\mathbb{C}}\{v_1,\dots,v_{n+1+\ell_j-j}\}\supset U_j\supset \mathrm{span}_{\mathbb{C}}\{v_1,\dots,v_{\ell_j-j}\}$. For every $j=1,\cdots, n$, we consider the space $M'_j=\mathrm{span}_{\mathbb{C}}\{v_{\ell_j-j+1},\dots,v_{n+1+\ell_j-j}\}$ (in particular,
$\dim M'_j=n+1$ and $U'_j\subset M'_j$) and we define the linear map $f_j:M'_j\to M'_{j+1}$ by $f_jv_a=v_a$, if 
$v_a\in M'_{j+1}$ and $f_jv_a=0$, otherwise. The representation $((M'_j)_{j=1}^n,\ (f_j)_{j=1}^{n-1})$ of $Q$ is isomorphic to $M^\bi$. In particular, the quiver Grassmannian $\Gr_\mathbf{e}(M^\bi)$ is isomorphic to the variety of collections $(U_j')_{j=1}^n$ of subspaces of $\mathbb{C}^{n+1+d(\bi)}$ such that 
\begin{enumerate}
\item $U_j'\subseteq M_j'$;
\item $\dim U_j'=j$; 
\item $f_jU'_j\subset U'_{j+1}$
for all $j$.
\end{enumerate}
We denote by $\pi_j:\mathbb{C}^{n+1+d(\bi)}\rightarrow M'_j$ the canonical projection.
The inclusion map 
$$
\zeta: \prod_{j=1}^n {\Gr}_{j}(M_j)\rightarrow \prod_{j=1}^n {\Gr}_{\ell_j}(\mathbb{C}^{n+1+d(\bi)}):\; (U_j')\mapsto (\pi_j^{-1}(U_j'))
$$
restricts to the required isomorphism $\zeta':{\Gr}_\mathbf{e}(M^\bi)\rightarrow X_w$.
\end{proof}

To prove Theorem \ref{Thm:flagSchubert}, it suffices to apply the following proposition.

\begin{proposition}\label{prop:action} 
The action of $w_{\bi}$ on $\{1, \cdots, \ell_n\}$ is given by:
\begin{enumerate}
\item If $\ell_j = \ell_{j-1} + 1$, then $w_{\bi} (\ell_j) = h_j + (n-j+2)$.
\item If $\ell_j = \ell_{j-1} + 2$, then  $w_{\bi} (\ell_j-1) = h_j$ and  $w_{\bi} (\ell_j) = h_j + n+1$ .
\end{enumerate}
\end{proposition}
\begin{proof}
In the first case, $\ell_{j-1} = \ell_j - 1$, we have (since the $\ell_i$ are strictly increasing and each $w_{j}$  is a sequence of strictly increasing simple reflections):
$$
w_{\bi}(\ell_j) = w_{n} \cdots w_{j} w_{j-1} \cdots w_{1} (\ell_j)= w_{n} \cdots w_{j} w_{j-1} (\ell_j).
$$
We have $w_{j - 1}(\ell_j) = h_j+1$ and $w_{k}(h_j + (k-j+1)) = h_j + (k - j +2)$ for all $k \geq j$, the claim follows.
\par
In the second case, $\ell_{j-1} = \ell_j - 2$, we have
$$
w_{\bi}(\ell_j-1) = w_{n} \cdots w_{j} w_{j-1} \cdots w_{1} (\ell_j-1)= w_{n} \cdots w_{j} w_{j-1} (\ell_j -1).
$$
But then $w_{j-1}(\ell_j -1) = h_{j-1}+1 = h_j$ and $w_{k}(h_j) = h_j$ for all $k \geq j$. Further we have
$$
w_{\bi}(\ell_j) = w_{n} \cdots w_{j} w_{j-1} \cdots w_{1} (\ell_j)= w_{n} \cdots w_{j} w_{j-1} (\ell_j).
$$
But $w_{j -1}(\ell_j) = \ell_j$ and $w_{j+k}(\ell_j + k) = \ell_j + k+1$ for all $k \geq 0$.
\end{proof}


\subsection{New gradings and filtrations}

Let $\mathfrak{sl}_{n+1}=\lie n^+\oplus \lie h \oplus \lie n^-$ be a fixed triangular decomposition of $\mathfrak{sl}_{n+1}$,
where $\lie h$ consists of diagonal matrices. Let $\alpha_1,\dots,\alpha_n$ be the simple roots of $\mathfrak{sl}_{n+1}$. Then the set of
positive roots is given by $\{\alpha_{p,q}=\alpha_p+\dots +\alpha_q,\ 1\le p\le q\le n\}$. We start with defining a grading on 
$\lie n^-$ associated to a fixed sequence $\mathbf{i}=(i_1,i_2,\cdots,i_k)\in\mathcal{D}$.
\par
For a positive root $\alpha=\alpha_{p,q}$ we write $f_{p,q}:=f_\alpha$. A degree function on $\lie n^-$ can be identified with a 
sequence $(t_{p,q})_{1\leq p\leq q\leq n}$ where $t_{p,q}$ is the degree of $f_{p,q}$.
We set $T_0=(t_{p,q}^0)_{1\leq p\leq q\leq n}$ where $t_{p,q}^0=q-p+1$ is the height of the root $\alpha_{p,q}$. For a number 
$l=1,2,\cdots,n-1$ we define a map $D_l$ sending a degree function to another by:
$$D_l(T)=T',$$
where $T=(t_{p,q})$ and $T'=(t_{p,q}')$ are degree functions satisfying
\begin{equation}
{t_{p,q}'}=\left\{\begin{matrix} t_{p,q}-1,& \text{ if }p\leq l<q;\\
t_{p,q}, & \text{otherwise}.\end{matrix}\right.
\end{equation}
We define $T^{\,\bi}=D_{i_k}\circ D_{i_{k-1}}\circ\cdots\circ D_{i_1}(T_0)$ and denote $T^{\,\bi}=(t_{p,q}^{\,\bi})$.
\par
\par
Consider the following grading on $\lie n^-$ defined by
$$\text{deg}_{\,\bi}(f_{p,q})=t_{p,q}^{\,\bi}.$$
\par
The following proposition is clear by definition.

\begin{proposition}\label{Prop:degree}
  The degree of $f_{p,q}$ is given by
  \[ \mathrm{\deg}_{\bi}( f_{p, q}) = \begin{cases}
       1 & \text{if } p = q;\\
       q - p + 1 - \# \{i_j \mid p \leq i_j < q\} & \text{ if } p\neq q.
     \end{cases} \]
\end{proposition}
As a direct consequence we have:
\begin{corollary}
  The Lie algebra $\lie n^-$ is filtered with respect to the grading
  $\deg_{ \bi}$.
\end{corollary}

Let $\lie n^{-, \bi}$ be the associated graded Lie algebra: the Lie algebra $\lie n^-$ is partially abelianized in $\lie n^{-, \bi}$. By Proposition \ref{Prop:degree}, we get the defining relations of $\lie n^{-, \bi}$ (for $p \leq s$):
\begin{equation}\label{Eq:Abel}
 [f_{p, q}, f_{s,r}] = \begin{cases}
     0 & \text{if } s \neq q + 1;\\
     0 & \text{if } s = q + 1 \text{ and }  \exists
     \hspace{0.75em} i_j = q;\\
     f_{p, r} & \text{else}.
   \end{cases}
\end{equation}
	
\bigskip

The grading on $\lie n^-$ induces a filtration $F_{\,\bi}$ on $U(\lie n^-)$ by letting
$$U_s(\lie n^-):=\text{span}\{x_1x_2\cdots x_t\,|\, x_j\in\lie n^-,\,\, \sum_{j=1}^t\text{deg}_{\,\bi}(x_j)\leq s\}.$$
We let ${\rm gr}_{F_{\,\bi}}U(\lie n^-)$ denote the associated graded algebra. Then it is clear that 
$${\rm gr}_{F_{\,\bi}}U(\lie n^-)\cong U(\lie n^{-,\bi})$$
is again an enveloping algebra.

Let $V(\lambda)$ be the irreducible representation of $\mathfrak{sl}_{n+1}$ of highest weight 
$\lambda\in\mathcal{P}^+$ and a highest weight vector $v_\lambda$. 
The filtration on $U(\lie n^-)$ induces a filtration on $V(\lambda)$ by defining 
$$V_s(\lambda):=U_s(\lie n^-).v_\lambda.$$
Let $V^{\,\bi}(\lambda)$ denote the associated graded vector space and let $v_\lambda^\bi$ 
be the image of $v_\lambda$ in $V^{\,\bi}(\lambda)$. 
It is clear that $V^{\,\bi}(\lambda)$ is a cyclic $U(\lie n^{-,\bi})$-module generated by $v_\lambda^{\,\bi}$.


\subsection{Realization as highest weight orbits}

Let $N^{\,\bi}$ be the connected linear unipotent algebraic group having $\lie n^{-,\bi}$ as Lie algebra. Then $N^{\,\bi}$ acts on 
$V^\bi(\lambda)$ and we define the closure of the highest weight orbit by
$$\ff^{\,\bi}(\lambda):=\overline{N^{\,\bi}\cdot[v_\lambda^{\,\bi}]}\subset \mathbb{P}(V^{\,\bi}(\lambda)).$$
In fact,
$$N^{\,\bi}\cdot[v_\lambda^{\,\bi}]=\left\{\left.\exp \left(\sum_{\alpha\in\Delta_+}c_\alpha f_\alpha\right)\cdot[v_\lambda^{\,\bi}]\,\right|\,c_\alpha\in\mathbb{C}\right\}.$$

For a projection sequence $\bi$ let $V_{w_{\bi}}(\lambda_{\bi})$ be the Demazure module inside $V(\lambda_{\bi})$, corresponding
to the Weyl group element  $w_{\bi}$. In more details, let
$\mathfrak{sl}_{ \bi} : = \mathfrak{sl}_{n+1+d(\bi)}$ and let 
$$
\mathfrak{sl}_{\bi} = \lie b_{\bi}^+ \oplus \lie n^-_{\bi} = \lie n_{\bi}^+
\oplus \lie h_{\bi} \oplus \lie n_{\bi}^-
$$ 
be the triangular decomposition. Then $V_{w_{\bi}}(\lambda_{\bi})$ is a cyclic $\lie b_{\bi}^+$
module inside $V(\lambda_{\bi})$ with the cyclic vector of weight $w_{\bi}(\lambda_{\bi})$. 
The main result of this subsection is the following theorem.

\begin{theorem}\label{Thm:HighestOrbit}
\begin{enumerate}
\item For any $\lambda\in\mathcal{P}^+$, $V^{\,\bi}(\lambda)\cong V_{w_{\bi}}(\lambda_{\bi})$ as $\lie n^{-,\bi}$-modules. 
\item For a regular weight $\lambda\in\mathcal{P}^+$, there exists an isomorphism of projective varieties
$$X_{w_{\,\bi}}\stackrel{\sim}{\longrightarrow} \ff^{\,\bi}(\lambda).$$
\end{enumerate}
\end{theorem}

Combining Theorem \ref{Thm:HighestOrbit} with Theorem \ref{Thm:flagSchubert}, we conclude that
a PBW linear degenerate flag variety is the closure of the highest weight orbit.

\begin{corollary}
For a regular weight $\lambda\in\mathcal{P}^+$, there exists an isomorphism of projective varieties
$$\ff_{n+1}^\bi\stackrel{\sim}{\longrightarrow} \ff^{\,\bi}(\lambda).$$
\end{corollary}

The strategy of the proof of Theorem \ref{Thm:HighestOrbit} is the following: the second part is a corollary 
of the first part. To prove the first part of the theorem, we first define the $\lie n^{-,\bi}$-module 
structure on $V_{w_{\bi}}(\lambda_{\bi})$, then we show that both $V^\bi(\lambda)$ and 
$V_{w_{\bi}}(\lambda_{\bi})$ have the same dimension (Section~\ref{dim-subsection}), finally we 
construct an explicit map between the spaces and prove its surjectivity (Section~\ref{map-subsection}).

\subsection{Proof of Theorem~\ref{Thm:HighestOrbit}: Dimension}\label{dim-subsection}
Here we first provide a $\lie n^{-,\bi}$-module structure on $V_{w_{\bi}}(\lambda_{\bi})$.  
We denote the set of positive roots of $\mathfrak{sl}_{ \bi}$ by $\Delta^+_{ \bi}$. Recall that the 
Weyl group $W_\bi$ acts on $\mathfrak{sl}_{\bi}$ by conjugation.
\par
The $\lie n^{-,\bi}$-module structure on $V_{w_{\bi}}(\lambda_{\bi})$ is defined by: for any $1\leq p\leq q\leq n$ and $v\in V_{w_{\bi}}(\lambda_{\bi})$,
$$f_{p,q}\cdot v:=w_\bi^{-1}f_{\ell_p,\ell_q}w_\bi\cdot v.$$
We will see below that $f_{p,q}\mapsto f_{\ell_p,\ell_q}$ is a  morphism of Lie algebras 
and $w_\bi^{-1}f_{\ell_p,\ell_q}w_\bi\in \lie b_{\bi}^+$ hence the module structure is well-defined.

\bigskip

Recall that $\ell_k : = h_k + k$. Let $\lie n^-_{w_\bi}$ denote the Lie subalgebra of $\mathfrak{sl}_{\bi}$ spanned by the root vectors 
$$\{f_{\ell_p,\ell_q}\,|\,1\leq p\leq q\leq n\}.$$

\begin{proposition}\label{Lem:eta}
The linear map $\eta:\lie n^{-, \bi}\ra \lie n_{w_\bi}^-$ given by   
$$f_{p, q} \mapsto f_{\ell_p, \ell_q}$$
is an isomorphism of Lie algebras.
\end{proposition}

\begin{proof}
We have to compare the commutator between $f_{\ell_p,\ell_q}$ and $f_{\ell_r,\ell_s}$ with (\ref{Eq:Abel}). We have for $p \leq s$
$$
[f_{\ell_p, \ell_q}, f_{\ell_s, \ell_r}] = \begin{cases} 0 & \text{ if } \ell_s\neq \ell_q+1 \\
f_{\ell_p, \ell_r} & \text{ else} 
\end{cases}
$$
So this is non-zero if and only if $\ell_q +1 = \ell_s$. Writing $\ell_q = h_k + k$, this implies that $h_{k+1} + (k+1) = h_k + k$, hence $h_{k+1} = h_k$ and therefore $k \notin \bi$. 
Compared to (\ref{Eq:Abel}), these are exactly the same relations as for $\lie n^{-, \bi}$ via the map $f_{p, q} \mapsto f_{\ell_p, \ell_q}$.
\end{proof}

We consider the following subspace of $\lie b_{ \bi}^+$:
$$\lie n^+_{w_\bi} : = {\sspan}_{\mathbb{C}}\{ e_{\alpha}\,|\,\alpha \in \Delta^+_{ \bi}, w_{ \bi}^{-1}(\alpha) < 0 \} \subset \lie b_{ \bi}^+.$$
Our goal is to show that $\lie n_{w_\bi}^+=w_\bi\, \lie n_{w_\bi}^- w_{\bi}^{-1}$.
For the proof we need the following lemma.

\begin{lemma}\label{hq}
Let $1\leq t\leq \ell_q$ be an integer satisfying 
$$h_q+1\leq w_{q-1}\cdots w_1(t)\leq h_q+q.$$
Then there exists $1\leq p\leq q$ such that $t=\ell_p$.
\end{lemma}
\begin{proof}
The proof is executed by induction on $q$. The case $q=1$ is clear since $h_1=0$ and $t=1=\ell_1$.
\par
There are two cases to examine:
\begin{enumerate}
\item Assume that $w_{q-1}\cdots w_1(t)=\ell_q$. By definition, either $h_{q-1}+1=h_q$ or $h_{q-1}=h_q$. 
\par
In the first case, $\ell_{q-1}+1<\ell_q$, hence by (\ref{Eq:w_k}) $w_{q-1}\cdots w_1(\ell_q)=\ell_q$ and $t=\ell_q$. 
\par
In the second situation, $\ell_{q-1}+1=\ell_q$, then $w_{q-1}^{-1}(\ell_q)=w_{q-1}^{-1}(\ell_{q-1}+1)=\ell_{q-1}$. By (\ref{Eq:w_k}), $w_{q-2}\cdots w_1(t)=\ell_{q-1}$. Applying the induction hypothesis gives some $p\leq q-1$ such that $t=\ell_p$.
\item Assume that $w_{q-1}\cdots w_1(t)\neq\ell_q$, i.e., $h_q+1\leq w_{q-1}w_{q-2}\cdots w_1(t)\leq h_q+q-1.$ We separate the proof into two cases as above.
\par
If $h_{q-1}+1=h_q$, the inequality reads $h_{q-1}+2\leq w_{q-1}w_{q-2}\cdots w_1(t)\leq h_{q-1}+q.$ By (\ref{Eq:w_k}), $h_{q-1}+1\leq w_{q-2}\cdots w_1(t)\leq h_{q-1}+q-1$. The induction hypothesis provides some $p\leq q-1$ such that $t=\ell_p$.
\par
If $h_{q-1}=h_q$, the inequality turns out to be $h_{q-1}+1\leq w_{q-1}w_{q-2}\cdots w_1(t)\leq h_{q-1}+q-1.$ Then either $h_{q-1}+1\leq w_{q-2}\cdots w_1(t)\leq h_{q-1}+q-2$ or $w_{q-2}\cdots w_1(t)=h_{q-1}+q$. We can apply the induction hypothesis in the first case to obtain some $p\leq q-1$ such that $t=\ell_p$. For the second case, since $h_{q-1}+q=\ell_{q-1}+1$ and $w_{q-2}\cdots w_1(\ell_{q-1}+1)=\ell_{q-1}+1$, $t=\ell_{q-1}+1=\ell_q$.
\end{enumerate}
\end{proof}

\begin{proposition}\label{Prop:conjugate}
We have $\lie n_{w_\bi}^+=w_\bi\, \lie n_{w_\bi}^- w_{\bi}^{-1}$, hence $\lie n_{w_\bi}^+$ is a Lie subalgebra of $\lie b_{\bi}^+$ isomorphic to $\lie n_{w_\bi}^-$ via $w_{\bi}$-conjugation.
\end{proposition}

\begin{proof}
By the definition of $\lie n_{w_\bi}^+$, it suffices to show that
$$w_\bi^{-1}(\Delta_\bi^+)\cap\Delta_\bi^-=\{-\alpha_{\ell_p,\ell_q}\,|\,1\leq p\leq q\leq n\}.$$
It is equivalent to show that for $\alpha_{r,s}$, $\alpha_{t,u}\in\Delta_\bi^+$, $w_\bi^{-1}(\alpha_{r,s})=-\alpha_{t,u}$ if and only if there exist $1\leq p\leq q\leq n$ such that $t=\ell_p$ and $u=\ell_q$.
\par
We suppose $w_\bi(\alpha_{t,u})\in\Delta_\bi^-$ and prove that $t,u\in\{\ell_1,\ell_2,\cdots,\ell_n\}$.
\par
Recall the definition of $w_\bi$: $w_\bi=w_nw_{n-1}\cdots w_1$, where
$$w_k=s_{h_k+1}\cdots s_{h_k+k-1}s_{\ell_k}.$$
It is clear that for any $k=1,\cdots,n$,
\begin{equation}\label{Eq:w_k}
w_k(r)=\begin{cases} r+1 & \text{ if } r=h_k+1,\cdots,h_k+k; \\
h_k+1 & \text{ if } r=\ell_k+1; \\
r & \text{ else}.
\end{cases}
\end{equation}
Writing $\alpha_{i,j}=\ve_i-\ve_{j+1}$, it is a consequence of (\ref{Eq:w_k}) that $w_k(\alpha_{i,j})\in\Delta_\bi^-$ if and only if $h_k+1\leq i\leq h_k+k$ and $j=\ell_k$.
We first show that there exists $q$ such that $u=\ell_q$. Since $\alpha_{t,u}\in\Delta_\bi^+$, there exists a smallest integer $q$ such that
$$w_{q-1}\cdots w_1(\alpha_{t,u})\in\Delta_\bi^+\ \ \text{but}\ \ w_qw_{q-1}\cdots w_1(\alpha_{t,u})\in\Delta_\bi^-.$$
It implies that $w_{q-1}\cdots w_1(\alpha_{t,u})=\alpha_{t',\ell_q}$ for some $h_q+1\leq t'\leq h_q+q=\ell_q$.
As $w_{q-1}\cdots w_1(\ell_q+1)=\ell_q+1$, hence $u=\ell_q$.

Using Lemma \ref{hq} we conclude that 
$w_\bi^{-1}(\Delta_\bi^+)\cap\Delta_\bi^-\subset\{-\alpha_{\ell_p,\ell_q}\,|\,1\leq p\leq q\leq n\}$. 
To finish the proof of the proposition, it suffices to show that both sides have the same cardinality: 
since the expression $w_\bi=w_nw_{n-1}\cdots w_1$ of $w_\bi$ is reduced, 
$$\# w_\bi^{-1}(\Delta_\bi^+)\cap\Delta_\bi^-=\ell(w_\bi)=\frac{n(n+1)}{2}=\# \{-\alpha_{\ell_p,\ell_q}\,|\,1\leq p\leq q\leq n\}.$$
\end{proof}

We see immediately

\begin{corollary}
The element $w_{ \bi} \in W_{ \bi}$ is triangular (in the sense of \cite{Fou1}).
\end{corollary}

\begin{proposition}[{\cite{Fou1}}]\label{Prop:Triangular} 
For any dominant weight $\lambda \in   \mathcal{P}^+$, we have $\dim V_{w_{\bi}}(\lambda_{\bi}) = \dim V^{ \bi} (\lambda)$.
\end{proposition}

\begin{proof}
A basis of $V_{w_{\bi}}(\lambda_{\bi})$ is parametrized by the lattice points of the marked chain polytope 
$P_{w_\bi}(\lambda_{\bi})$ (by {\cite[Corollary 2]{Fou1}}). But the underlying poset is isomorphic 
to the poset induced from $\lie n^-$ (see \cite{FeFoLit}), this implies that the polytopes are isomorphic. 
Hence the numbers of lattice points are equal and so the dimensions coincide.
\end{proof}

\subsection{Proof of Theorem~\ref{Thm:HighestOrbit}: Explicit map}\label{map-subsection}
As a Demazure module, $V_{w_{\bi}}(\lambda_{\bi})$ is a $\lie b_{\bi}^+$-module,
hence an $\lie n_{w_\bi}^+$-module by restriction. Let $v_{w_\bi}$ be a lowest weight vector in $V_{w_{\bi}}(\lambda_{\bi})$ such that 
$$V_{w_{\bi}}(\lambda_{\bi})=U(\lie b_\bi^+)v_{w_\bi}.$$

\begin{lemma}\label{bn}
We have $U(\lie b_\bi^+)v_{w_\bi}=U(\lie n_{w_\bi}^+)v_{w_\bi}$.
\end{lemma}

\begin{proof}
We take $e_\alpha\in\lie b_{\bi}^+$ for some $\alpha\in\Delta_\bi^+$ such that $e_\alpha\cdot v_{w_\bi(\lambda_\bi)}\neq 0$. This condition implies that 
$$w_\bi\, e_{w_\bi^{-1}(\alpha)}\cdot v_{\lambda_\bi}\,w_\bi^{-1}\neq 0,$$
i.e. $e_{w_\bi^{-1}(\alpha)}v_{\lambda_\bi}\neq 0$. Since $v_{\lambda_\bi}$ is a highest weight vector, $w_\bi^{-1}(\alpha)<0$.
\end{proof}

By Proposition \ref{Prop:conjugate},
$$U(\lie n_{w_\bi}^+)v_{w_\bi}=w_\bi\,U(\lie n_{w_\bi}^-)w_\bi^{-1}w_\bi\, v_{\lambda_{\bi}}w_\bi^{-1}=w_\bi \,U(\lie n_{w_\bi}^-)v_{\lambda_{\bi}}w_\bi^{-1}$$
where $v_{\lambda_{\bi}}$ is a highest weight vector in $V(\lambda_{\bi})$. Hence as $\lie n_{w_\bi}^-$-modules, we have $V_{w_{\bi}}(\lambda_{\bi})=w_\bi \,U(\lie n_{w_\bi}^-)v_{\lambda_{\bi}}w_\bi^{-1}$.

\bigskip

Now we prove Theorem \ref{Thm:HighestOrbit}. The proof consists of two parts: the first and crucial step is to consider the case of fundamental $\lambda$ 
and the second step is to deduce the general case. We prepare the following lemma.

\begin{lemma}\label{Prop:twist}
We have
$$w_\bi^{-1}V_{w_{\bi}}(\varpi_{\ell_r})w_\bi  = U(\lie n^-_{w_{\bi}} \cap \lie n^-_{\ell_r}).v_{\varpi_{\ell_r}}$$
where $\lie n^-_{\ell_r}$ is the nilpotent radical corresponding to the weight $\varpi_{\ell_r}$ 
(i.e., generated by all $f_{p,q}$ with $p \leq \ell_r \leq q$).
\par
Moreover, if $\{p_1, \cdots, p_s\}$ and $\{q_1, \cdots, q_s\}$ are two sets of pairwise distinct 
indices with $p_i \leq r \leq q_i$ for all $i$, then
$$ \prod_{i} f_{p_i, q_i}\cdot v_1 \wedge \cdots \wedge v_r \neq 0 \in V^{\bi}(\varpi_r).$$
\end{lemma}
\begin{proof} 
We have shown in the argument after Lemma \ref{bn} that 
$$U(\lie n^-_{w_\bi})v_{\varpi_{\ell_r}}=w_\bi^{-1}V_{w_\bi}(\varpi_{\ell_r})w_\bi.$$
It is enough to prove the following two claims. Firstly,
$$\deg_\bi(f_{p,q})+\deg_\bi(f_{t,s})=\deg_\bi(f_{p,s})+\deg_\bi(f_{t,q})$$
for all $p < t \leq s < q$. This follows easily from Proposition~\ref{Prop:degree}. 
\par
Secondly, we need for $p<s<q$, that
$$\deg_\bi(f_{p,q})\leq \deg_\bi(f_{p,s-1})+\deg_\bi(f_{s,q})$$
but this follows since the degree function is convex by definition.
\par
The second claim shows that the module is generated by the nilpotent radical,
while the first claim shows that for a fixed weight, all monomials in root vectors of the nilpotent 
radical have the same degree (they are parametrized by elements of the symmetric group 
$\mathfrak{S}_k$ (for some $k$)). But $\mathfrak{S}_k$ is generated by transpositions and hence 
it suffices to note that the degree is not changed under transpositions.
\end{proof}

We define a linear isomorphism
$$\zeta:V^{\bi}(\varpi_r)\ra w_\bi^{-1}V_{w_{\bi}}(\varpi_{\ell_r})w_\bi.$$
Consider the map $\Lambda^r \mathbb{C}^{n+1} \rightarrow \Lambda^r \mathbb{C}^{\ell_n + 1}$ induced by 
$$
v_i \mapsto \begin{cases} v_{\ell_i } & \text{ if } i \leq r \\ v_{{\ell_{i-1}}+1} & \text{ if } i > r \end{cases}
$$
Let $\{c_1,c_2,\cdots,c_{\ell_r-r}\}=\{1,2,\cdots,\ell_r\}\backslash\{\ell_1,\ell_2,\cdots,\ell_r\}$. 
We consider the map  
$$\Lambda^r \mathbb{C}^{\ell_n + 1}\ra \Lambda^{\ell_r} \mathbb{C}^{\ell_n + 1},\ \ 
u\mapsto u \wedge v_{c_1} \wedge \cdots \wedge v_{c_{\ell_r - r}}.$$
The composition $\Lambda^r \mathbb{C}^{n+1} \rightarrow \Lambda^r \mathbb{C}^{\ell_n + 1}\ra \Lambda^{\ell_r} \mathbb{C}^{\ell_n + 1}$
of these two linear maps will be denoted by $\zeta$.
\bigskip\\ \noindent
\textbf{Claim.} $\zeta$ induces a linear isomorphism 
$$\zeta:V^{\bi}(\varpi_r)\ra w_\bi^{-1}V_{w_{\bi}}(\varpi_{\ell_r})w_\bi.$$
\begin{proof}
By definition $\zeta$ is injective. By Proposition \ref{Prop:Triangular}, it suffices to show that 
its image is contained in $w_\bi^{-1}V_{w_{\bi}}(\varpi_{\ell_r})w_\bi$. By Lemma \ref{Prop:twist}, 
we can easily see that $ w_\bi^{-1}V_{w_{\bi}}(\varpi_{\ell_r})w_\bi$ is spanned by the vectors 
$v_{j_1} \wedge \cdots \wedge v_{j_{\ell_r}} $ such that
$$\{c_1, \cdots, c_{\ell_r - r } \} \subset \{j_1, \cdots, j_{\ell_r}\} 
\text{ and if } j_t > \ell_r \text{ then } j_{t} = \ell_s + 1 \text{ for some } s.$$
So we can write its generators as
\begin{equation}\label{Eq:e}
\underline{v} =  v_{\ell_{t_1}} \wedge \cdots \wedge v_{\ell_{t_s}} \wedge v_{\ell_{t_{s+1}} + 1} 
\wedge \cdots \wedge v_{\ell_{t_r} + 1} \wedge v_{c_1} \wedge \cdots \wedge v_{c_{\ell_r - r}}
\end{equation}
where $1\leq t_1<\cdots<t_r\leq n$ and $\ell_{t_s} \leq \ell_r < \ell_{t_{s+1}}$. 
By definition, $\zeta(v_{i_1}\wedge\cdots\wedge v_{i_r})$ is of this form.
\end{proof}

\begin{lemma}\label{Prop:fund}
For any $r=1,2,\cdots, n$, $V^{\bi}(\varpi_r)\cong V_{w_{\bi}}(\varpi_{\ell_r})$ as $\lie n^{-,\bi}$-modules.
\end{lemma}
\begin{proof}
Recall the isomorphism of Lie algebras from Proposition \ref{Lem:eta}:
$$\eta:\lie n^{-, \bi}\ra \lie n_{w_\bi}^-,\ \ f_{p,q} \mapsto f_{\ell_p, \ell_q}.$$ 
It suffices to show that for any $f_{p,q}\in \lie n^{-,\bi}$,
\begin{equation}\label{Eq:module}
\eta(f_{p,q})\cdot \zeta(v_{i_1} \wedge \cdots \wedge v_{i_r}) = \zeta(f_{p,q}\cdot v_{i_1} \wedge \cdots \wedge v_{i_r}),
\end{equation}
which implies that $\zeta$ is an isomorphism of $\lie n^{-, \bi}$-modules.
\par
We fix some notations. By Lemma \ref{Prop:twist}, there exist two sets of pairwise distinct 
indices $\{p_1,p_2,\cdots,p_s\}$ and $\{q_1,q_2,\cdots,q_s\}$ with $p_i\leq r\leq q_i$ for $i=1,2,\cdots,s$ such that
$$v_{i_1} \wedge \cdots  \wedge v_{i_r} = \prod_{i} f_{p_i, q_i}\cdot v_1 \wedge \cdots \wedge v_r\in V^\bi(\varpi_r),$$
where $f_{p,q}\in\lie n^{-,\bi}$. We suppose that for some $t$, $i_t\leq r<i_{t+1}$
$$\{p_1,p_2,\cdots,p_{r-t}\}=\{1,2,\cdots,r\}\backslash\{i_1,\cdots,i_t\}$$
and
$$\{q_1,q_2,\cdots,q_{r-t}\}=\{i_{t+1}-1,\cdots,i_r-1\}.$$
We consider $\underline{v}$ as in (\ref{Eq:e}): let 
$$\{m_1,m_2,\cdots,m_{r-s}\}=\{1,2,\cdots,r\}\backslash\{t_1,t_2,\cdots,t_s\},$$
then
$$\underline{v} = \pm\prod_{i=1}^{r-s} f_{\ell_{m_i}, \ell_{t_{s+i}}}\cdot v_1 \wedge \cdots \wedge v_{\ell_r}.$$
Notice that $t$ satisfies $i_t\leq r<i_{t+1}$, then $\zeta(v_{i_1} \wedge \cdots  \wedge v_{i_r})$ reads
\begin{equation}\label{Eq:zetaimage}
\underline{v}'=v_{\ell_{i_1}}\wedge\cdots\wedge v_{\ell_{i_t}}\wedge v_{\ell_{i_{t+1}-1}+1}\wedge
\cdots\wedge v_{\ell_{i_{r}-1}+1}\wedge v_{c_1}\wedge\cdots\wedge v_{c_{\ell_r-r}},
\end{equation}
which is
$$\pm\prod_{i=1}^{r-t} f_{\ell_{p_i}, \ell_{q_i}}\cdot v_1 \wedge \cdots \wedge v_{\ell_r}.$$

The proof of (\ref{Eq:module}) is separated into three cases. We give the statement of the case 
$p\leq q< r$, the cases $p\leq r\leq q$ and $r<p\leq q$ are similar. The proofs can be done by direct computations, we omit them here.

\par
The following statements are equivalent:
\begin{itemize}
\item $f_{p,q}\cdot v_{i_1} \wedge \cdots \wedge v_{i_r}=0$ in $V^\bi(\varpi_r)$;
\item $p\in\{p_1,\cdots,p_{r-t}\}$ or $q+1\notin\{p_1,\cdots,p_{r-t}\}$ or $q\in\bi$;
\item $f_{\ell_p,\ell_q}\cdot \zeta(v_{i_1} \wedge \cdots \wedge v_{i_r})=0$.
\end{itemize}

Now to verify (\ref{Eq:module}), it suffices to consider the case 
$f_{p,q}\cdot v_{i_1} \wedge \cdots \wedge v_{i_r}\neq 0$. In this case we may suppose that $p=i_w$ 
for some $w\leq t$; since $q+1\leq r$, the right hand side of (\ref{Eq:module}) reads:
$$v_{\ell_{i_1}}\wedge\cdots \wedge v_{\ell_{i_{w-1}}}\wedge v_{\ell_{q+1}}\wedge v_{\ell_{i_{w+1}}}\wedge\cdots\wedge 
v_{\ell_{i_{t+1}-1}+1}\wedge\cdots\wedge v_{\ell_{i_{r}-1}+1}\wedge v_{c_1}\wedge\cdots\wedge v_{c_{\ell_r-r}},$$
while the left hand side reads
$$v_{\ell_{i_1}}\wedge\cdots \wedge v_{\ell_{i_{w-1}}}\wedge v_{\ell_{q}+1}\wedge v_{\ell_{i_{w+1}}}
\wedge\cdots\wedge v_{\ell_{i_{t+1}-1}+1}\wedge\cdots\wedge v_{\ell_{i_{r}-1}+1}\wedge v_{c_1}\wedge\cdots\wedge v_{c_{\ell_r-r}}.$$
We know that the hypothesis $f_{p,q}\cdot v_{i_1} \wedge \cdots \wedge v_{i_r}\neq 0$ implies 
$q\notin \bi$ and hence $\ell_q+1=\ell_{q+1}$, which proves (\ref{Eq:module}).
\end{proof}

Now we turn to the general case $\lambda=\sum_{i=1}^n\lambda_i\varpi_i$.
\begin{proof}
\begin{enumerate}
\item We consider the following commutative diagram:
\begin{tiny}
\[
\xymatrix{
w_\bi^{-1}V_{w_{\bi}}(\lambda_{\,\bi})w_\bi \ar[d]^-{\iota}  & V^{\,\bi}(\lambda) \ar[d]^-{\pi} \ar[l]_-{\beta}\\ 
(w_\bi^{-1}V_{w_{\bi}}(\Psi^{\bi}(\varpi_1))w_\bi)^{\ts\lambda_1}\ts\cdots\ts (w_\bi^{-1}V_{w_{\bi}}(\Psi^{\bi}(\varpi_n))w_\bi)^{\ts \lambda_n} \ar[r]^-{\vp} & V^{\,\bi}(\varpi_1)^{\ts \lambda_1}\ts\cdots \ts V^{\,\bi}(\varpi_n)^{\ts\lambda_n}.
}
\]
\end{tiny}
where 
\begin{itemize}
\item the map $\iota$ is an embedding of $w_\bi^{-1}\lie n_{w_{\bi}}^+w_\bi$-modules into the Cartan component;
\item the map $\pi$ is a $\lie n_{w_\bi}^-$-module projection onto the Cartan component;
\item the map $\vp$ is an isomorphism of $\lie n_{w_\bi}^-$-modules by Lemma~\ref{Prop:fund}.
\end{itemize}

Hence as composition, $\beta:=\iota^{-1}\circ \vp^{-1}\circ\pi$ is surjective. 
By Proposition \ref{Prop:Triangular}, $\beta$ is an isomorphism for dimension reasons. 
This completes the proof of Theorem~\ref{Thm:HighestOrbit} (1).
\item By definition, 
$$X_{w_\bi}:=\overline{B_\bi\cdot [v_{w_\bi(\lambda_\bi)}]}\subset \mathbb{P}(V_{w_\bi}(\lambda_\bi))\ \ \text{and}\ \ \ff^\bi(\lambda):=\overline{N^{-,\bi}\cdot [v_\lambda^{\bi}]}\subset \mathbb{P}(V^{\bi}(\lambda)).$$
We examine the space $B_\bi\cdot v_{w_\bi(\lambda_\bi)}$: for $\alpha\in\Delta_\bi^+$, let $U_\alpha$ denote the corresponding root subgroup in $B_\bi$, then
$$B_\bi\cdot v_{w_\bi(\lambda_\bi)}=\prod_{\alpha\in\Delta_\bi^+}U_\alpha\cdot v_{w_\bi(\lambda_\bi)}=\prod_{\alpha\in\Delta_\bi^+\, ,\, w_\bi^{-1}(\alpha)<0}U_\alpha\cdot v_{w_\bi(\lambda_\bi)}=N_\bi^+\cdot v_{w_\bi(\lambda_\bi)}.$$
Therefore $X_{w_\bi}=\overline{N_\bi^+\cdot [v_{w_\bi(\lambda_\bi)}]}$. By Theorem \ref{Thm:HighestOrbit} (1), conjugating by $w_\bi$ gives the desired isomorphism of projective varieties
$$X_{w_\bi}=\overline{N^+_{\bi}\cdot [v_{w_\bi(\lambda_\bi)}]}\cong \overline{w_\bi^{-1}N_\bi^+\cdot [v_{w_\bi(\lambda_\bi)}]w_\bi}=\overline{N_\bi^-\cdot[v_{\lambda_\bi}]}\cong \overline{N^{-,\bi}\cdot [v_\lambda^\bi]}=\ff^\bi(\lambda).$$
\end{enumerate}
\end{proof}

Fibers on the PBW locus $U_{\text{PBW}}$ share the geometric properties of Schubert varieties.

\begin{corollary}
For any $\bi\in\mathcal{D}$, $\ff_{n+1}^\bi$ is a normal variety having rational singularities; it is Cohen-Macaulay and Frobenius split.
\end{corollary}


\subsection{Construction of a unipotent group scheme}
We close this section with the construction a flat unipotent group scheme acting on the fibers over
the PBW locus with dense orbits. This scheme can be regarded as a "universal" version of the construction given above;
we use the transversal slices introduced in subsection \ref{Ts}.
 
\begin{theorem} There exists a flat unipotent group scheme $\Gamma_{\rm PBW}\rightarrow T_{\rm PBW}$ 
acting on $\pi^{-1}(T_{\rm PBW})\rightarrow T_{\rm PBW}$ with a dense orbit.
\end{theorem}

\begin{proof} We start with the trivial group scheme $G\times R\rightarrow R$ and consider the closed subscheme ${\rm Aut}=s^{-1}(\Delta)$, 
where $s:G\times R\rightarrow R\times R$ is the shear map, and $\Delta\subset R\times R$ is the diagonal. Then the fiber of ${\rm Aut}$ over $M$ is just the automorphism group of $M$ considered as a representation, which acts on the fiber $\pi^{-1}(M)$. Note that ${\rm Aut}$ is not flat since the dimension of automorphism group varies with $M$.  We would like to construct closed subgroup schemes of ${\rm Aut}$. Restricting to $U_{\rm PBW}$, we can, without loss of generality, work over the transversal slice $T_{\rm PBW}$. We identify $T_{\rm PBW}$ with ${\mathbb C}^{n-1}$ (since an element of $T_{\rm PBW}$ is specified by parameters ${\bf\lambda}=(\lambda_{11},\ldots,\lambda_{n-1,n-1})$) and define $\Gamma\subset G\times T_{\rm PBW}$ as the closed subscheme of all tuples
$((g_1({\bf\lambda}),\ldots,g_n({\bf\lambda})),{\bf\lambda})$ such that
$$(g_i({\bf\lambda}))_{p,q}=\left\{\begin{array}{ccc}0&,&p<q,\\ 1&,&p=q,\\ \lambda_{q-1,q-1}x_{p,q}&,&i<q<p,\\
x_{p,q}&,&q\leq i<p,\\
\lambda_{p-1,p-1}x_{p,q}&,&q<p\leq i\end{array}\right.$$
for arbitrary $(x_{p,q})_{p>q}$. 

It is immediately verified that $\Gamma$ is a flat unipotent closed subgroup scheme of ${\rm Aut}$ over $T_{\rm PBW}$. To see that it acts with an open orbit on each fiber, one just verifies that the stabilizer of the standard flag $(\langle v_1\rangle,\langle v_1,v_2\rangle,\ldots)$ in $\Gamma$ is trivial.
\end{proof}

\section{Geometry of linear degenerations - the flat locus}\label{flat}

Since the orbit $\mathcal{O}_{{\bf r}^2}$ is minimal in the flat locus $U_{\rm flat}$, the linear degenerate flag variety ${\rm Fl}^{{\bf r}^2}(V)$ is maximally degenerated, thus we call it the maximally flat (mf)--linear degeneration of the flag variety.


\begin{theorem}\label{t1}
${\rm Fl}^{{\bf r}^2}(V)$ is of dimension ${n+1 \choose 2}$, and its irreducible components are naturally parametrized by non-crossing arc diagrams on $n$ points. Consequently, the number of irreducible components equals the $n$-th Catalan number.
\end{theorem}

An arc diagram on $n$ points is a subset $A$ of $\{(i,j),\, 1\leq i<j\leq n\}$ (draw an arc from $i$ to $j$ for every element $(i,j)$ of $A$). An arc diagram $A$ is called non-crossing if there is no pair of different elements $(i,j)$, $(k,l)$ in $A$ such that 
$i\le k<j\le l$ 
(that is, two arcs are not allowed to properly cross, or to have the same left or right point. But immediate succession of arcs, like for example $\{(1,2),(2,3)\}$, is allowed).

To a non-crossing arc diagram we associate a rank tuple ${\bf r}(A)$ by
$$r(A)_{i,j}=i-|\{\mbox{arcs in $A$ starting in $[1,i]$ and ending in $[i+1,j]$}\}|.$$
Define $S_A\subset {\rm Fl}^{{\bf r}^2}(V)$ as the set of all tuples $(U_1,\ldots,U_n)$ such that
$${\rm rank}((f_{j-1}\circ\ldots\circ f_i)|_{U_i}:U_i\rightarrow U_j)=r(A)_{i,j}$$
for all $i<j$.

Moreover, define representations $\overline{N}_A$ and $N_A$ of $Q$ by
$$\overline{N}_A=\bigoplus_{(i,j)\in A}U_{i,j-1},\; N_A=\bigoplus_i P_i^{c_i}\oplus \overline{N}_A,$$
where
$$c_i=1+|\{\mbox{arcs ending in $i$}\}|-|\{\mbox{arcs starting in $i$}\}|.$$
It is immediately verified that ${\bf r}(A)$ is precisely the rank tuple of $N_A$.

We have:

\begin{theorem}\label{t3} The irreducible components of ${\rm Fl}^{{\bf r}^2}(V)$ are the closures of the $S_A$, for $A$ a non-crossing arc diagram.
\end{theorem}

\subsection{Proofs of the theorems}\label{Sec:Proofs}

We can now combine the results and methods developed so far to give proofs of Theorems \ref{t1}, \ref{t2} and \ref {t3}.


To prove Theorem \ref{t1}, we consider $M^2=P\oplus X$ with $P=A$ and $X=S\oplus A^*/S$ and reformulate the criterion of Theorem \ref{tc}. Using the exact sequence
$$0\rightarrow S\rightarrow A^*\rightarrow A^*/S\rightarrow 0,$$
and injectivity of $A^*$, we can rewrite
$$\dim{\rm Hom}(\overline{N},S\oplus A^*/S)-\dim{\rm Hom}(\overline{N},A^*)=\dim{\rm Ext}^1(\overline{N},S).$$
We thus have to check the inequality
$$\dim{\rm End}(\overline{N})\geq\dim{\rm Ext}^1(\overline{N},S).$$
Writing
$$\overline{N}=\bigoplus_{1\leq i\leq j<n}U_{i,j}^{n_{i,j}},$$
we have
$$\dim{\rm Ext}^1(\overline{N},S)=\sum_{1\leq i\leq j<n}n_{i,j},$$
and certainly
$$\dim{\rm End}(\overline{N})\geq \sum_{1\leq i\leq j<n}n_{i,j}^2.$$
This proves the claim about the dimension of ${\rm Gr}_{\bf e}(M^2)$. The irreducible components are parametrized by the representations $N$ as above for which the direct summand $\overline{N}$ satisfies
$$\dim{\rm End}(\overline{N})=\dim{\rm Ext}^1(\overline{N},S).$$
To satisfy this equality, it is thus necessary and sufficient for $\overline{N}$ to have all multiplicities $n_{i,j}$ of indecomposables equal to either $0$ or $1$, and there should be no non-zero maps between those $U_{i,j}$ for which $n_{i,j}=1$. But this can be made explicit since
$$\dim{\rm Hom}(U_{i,j},U_{k,l})=1\mbox{ if }k\leq i\leq l\leq j,$$
and zero otherwise. Thus $\overline{N}$ has to be of the form
$$\overline{N}=\bigoplus_{(i,j)\in I}U_{i,j-1}$$
for a set $I$ of pairs $(i,j)$ with $i\leq j$, such that there is no pair of different elements $(i,j),(k,l)\in I$ 
fulfilling $i\leq k<j\leq l$. These are precisely the representations $\overline{N}_A$ associated to non-crossing arc diagrams introduced above. It suffices to check that these $\overline{N}$ fulfill the additional assumptions, that is, that they embed into $S\oplus A^*/S$ and the condition on dimension vectors. But this is easily verified.

We now turn to the proof of the first part of Theorem \ref{t2}. Suppose that $M$ does not degenerate to $M^2$. By Theorem \ref{t5}, $M$ is a degeneration of some $M({\bf a}^{i,j})$. We claim that ${\rm Gr}_{\bf e}(M({\bf a}^{i,j}))$ has dimension strictly bigger than $n(n+1)/2$. Namely, we choose $\overline{N}=S_i\oplus S_j$. The conditions of Theorem \ref{tc} are easily seen to be violated. By upper
semi-continuity of fiber dimensions, $\dim{\rm Gr}_{\bf e}(M)$ is also strictly bigger than $n(n+1)/2$. On the other hand, again applying semi-continuity of fiber dimensions, since ${\rm Gr}_{\bf e}(M^2)$ has the correct dimension, also ${\rm Gr}_{\bf e}(M')$ has dimension $n(n+1)/2$ for every representation $M'$ degenerating to $M^2$. But by the first part of Theorem \ref{ag}, the flat locus in $R$ is precisely the locus where the fibers have the minimal dimension. 

For the second part of Theorem \ref{t2} we argue similarly: Suppose that $M$ does not degenerate to $M^1$. By Theorem \ref{t5}, $M$ is a degeneration of some $M({\bf a}^{i})$. We claim that ${\rm Gr}_{\bf e}(M({\bf a}^{i}))$ is reducible. Namely, we consider the two subrepresentations $N_1$ and $N_2$ given by $N_1=A$ and $N_2=\bigoplus_{j\neq i} P_j\oplus S_i\oplus P_{i+1}$ so that $\overline{N_1}=0$ and $\overline{N_2}=S_i$ (notation as in section~\ref{Sec:QG}). Both $N_1$ and $N_2$ fulfill equality in the estimate of Theorem \ref{tc}, thus ${\rm Gr}_{\bf e}(M({\bf a}^i))$ has at least two irreducible components. We claim that ${\rm Gr}_{\bf e}(M)$ is reducible. Suppose that it is irreducible. We consider the subset $U\subset R_\mathbf{d}(Q)$ consisting of all representations degenerating to $M$: it is an irreducible open subset in $R_\mathbf{d}(Q)$. The restriction $\pi^{-1}(U)\rightarrow U$ of $\pi$ is $G_\mathbf{d}$-equivariant and flat, and the orbit of $M$ in $U$ is the only closed orbit. By the second part of Theorem \ref{ag}, ${\rm Gr}_{\bf e}(M({\bf a}^{i}))$ is irreducible, a contradiction.
\par
On the other hand, since ${\rm Gr}_{\bf e}(M^1)$ is irreducible, by the second part of Theorem \ref{ag}, ${\rm Gr}_{\bf e}(M')$ is irreducible for every representation degenerating to $M^1$.

\subsection{Geometric properties of the mf--linear degenerate flag variety}

\begin{theorem} The scheme ${\rm Fl}^{{\bf r}^2}(V)$ is reduced and locally  a complete intersection. Consequently, all linear degenerations of flag varieties over $U_{\rm flat}$ are reduced locally complete intersection varieties.
\end{theorem}
Let us first show that the scheme ${\rm Fl}^{{\bf r}^2}(V)$ is locally a complete intersection. By definition ${\rm Fl}^{{\bf r}^2}(V)={\Gr}_\mathbf{e}(M)$ where $M\simeq M^{2}$ is the $Q$--representation defined by a tuple $f_\ast=(f_1,\cdots, f_{n-1})$ such that $\mathbf{r}(f_\ast)=\mathbf{r}^2$ (see Definition~\ref{Def:R1R2}) and $\mathbf{e}=(1,2,\cdots, n)$. Consider the affine variety $\textrm{Hom}(\mathbf{e},M)$ consisting of tuples $((N_i), (g_i))$ inside the vector space $\mathcal{M}:=R_\mathbf{e}\times \prod_{i=1}^n {\rm Hom}_\mathbb{C}(\mathbb{C}^{i}, \mathbb{C}^{n+1})$ such that $f_{i+1}\circ g_i=g_{i+1}\circ N_i$ for every $i=1,\cdots, n-1$. The quiver Grassmannian ${\Gr}_\mathbf{e}(M)$ can be realized as a geometric quotient ${\Gr}_\mathbf{e}(M)\simeq \textrm{Hom}^0(\mathbf{e},M)/G_\mathbf{e}$ (see \cite{CFR1}) where $\textrm{Hom}^0(\mathbf{e},M)$ is the open subvariety of $\textrm{Hom}(\mathbf{e},M)$ consisting of points $((N_i), (g_i))$ such that all the maps $g_i$ are injective. It is hence enough to show that $\textrm{Hom}^0(\mathbf{e},M)$ is locally complete intersection. We already know that ${\rm Fl}^{{\bf r}^2}(V)$ is equidimensional of dimension $\frac{n(n+1)}{2}$ and hence $\textrm{Hom}^0(\mathbf{e},M)$ is equidimensional of dimension $\frac{n(n+1)}{2}+\textrm{dim}\, G_\mathbf{e}=\frac{n(n+1)(n+2)}{3}$. Its codimension in $\mathcal{M}$ is given by $\frac{n(n+1)^2}{2}$.
This is precisely the number of equations defining $\textrm{Hom}(\mathbf{e},M)$ inside $\mathcal{M}$ and hence $\textrm{Hom}^0(\mathbf{e},M)$ is locally a complete intersection. 

Now that we know that the scheme ${\rm Fl}^{{\bf r}^2}(V)$ is locally a complete intersection, once we prove that it is also generically reduced, we can apply the third part of Theorem \ref{ag} to conclude that it is reduced. 

We hence prove that ${\rm Fl}^{{\bf r}^2}(V)$ is generically reduced.  For this, we first consider duality of non-crossing arc diagrams.

Let $A$ be an arc diagram as above. A pair $(i,j)$ with $i\leq j$ of indices is called a {\it chain} if there is a sequence of arcs $$(i=i_0,i_1),(i_1,i_2),\ldots,(i_{k-1},i_k=j)$$ in $A$. 
It is called a {\it complete chain} if it is a chain, and there is neither an arc ending in $i$ nor an arc starting in $j$. In particular, an isolated vertex $i$ (that is, $i$ 
is not connected to any arc) counts as a complete chain $(i,i)$ of length $0$.

Define $A^*$, the dual of $A$, as
$$A^*=\{(i-1,j)\mid i\geq 2\mbox{ and }(i,j)\mbox{ is a complete chain in }A\}.$$
We denote the map $A\mapsto A^*$ by $*$.
Let us also denote by ${\rm op}$ the symmetry on the arc diagrams induced by the Dynkin diagram symmetry $i\mapsto n+1-i$.
\begin{lemma}
The map $A\mapsto {\rm op}A^*$ is an involution. 
\end{lemma}
\begin{proof}
We need to show that ${\rm op} (*^{-1}A)=({\rm op}A)^*$. Let $$(i=i_0,i_1),(i_1,i_2),\ldots,(i_{k-1},i_k=j)$$ be a complete chain of arcs in $A$.
After applying the composite map $*op$ to this chain we obtain a long arc $(n-j,n+1-i)$ and $k$ complete chains of the form
\[
(n+1-i_{a+1},n+1-i_{a+1}+1),(n+1-i_{a+1}+1,n+1-i_{a+1}+2),\dots, (n-i_a-1,n-i_a),  
\] 
$a=0,\dots,k-1$.
We conclude that $({\rm op}A)^*$ consists of the parts described above (one part for each complete chain in $A$) . Now it suffices to note that
${\rm op} (*^{-1}A)$ consists of the same parts.
\end{proof}

Using the self-duality of the representation $M^2$ under the previous symmetry, we can define quotient representations $Q_A$ of $M^2$ dually to the subrepresentations $N_A$ of $M^2$ for every non-crossing arc diagram $A$. More precisely, we have
$$\overline{Q}_A=\bigoplus_{(i,j)\in A}U_{i+1,j},\; Q_A=\overline{Q}_A\oplus\bigoplus_i I_i^{d_i},$$
where
$$d_i=1+|\{\mbox{arcs starting in $i$}\}|-|\{\mbox{arcs ending in $i$}\}|.$$

The following is then proved by an explicit construction:

\begin{proposition} For every arc diagram $A$, there exists a short exact sequence
$$0\rightarrow N_A\rightarrow M^2\rightarrow Q_{A^*}\rightarrow 0.$$
We have $$\dim{\rm Hom}(N_A,Q_{A^*})=n(n+1)/2.$$
\end{proposition}

\begin{proof} We consider the direct sum of the following short exact sequences:
\begin{itemize}
\item One copy of
$$0\rightarrow U_{i,j-1}\rightarrow I_{j-1}\oplus S_i\rightarrow I_i\rightarrow 0$$
for every arc $(i,j)$ in $A$,
\item one copy of
$$0\rightarrow P_j\rightarrow P_i\oplus S_j\rightarrow U_{i,j}\rightarrow 0$$
for every complete chain from $i$ to $j$ in $A$ (equivalently, for every arc $(i+1,j)$ in $A^*$),
\item one copy of
$$0\rightarrow P_j\rightarrow P_j\rightarrow 0\rightarrow 0$$
for every $j$ such that there exists an arc $(i,j)$ in $A$,
\item one copy of
$$0\rightarrow 0\rightarrow I_i\rightarrow I_i\rightarrow 0$$
for every $i$ such that there is no arc $(j,i+1)$ in $A$.
\end{itemize}

Using the definition of $N_A$, $M^{(2)}$ and $Q_{A^*}$, as well as the definition of $A^*$ and the fact that $A$ is non-crossing, one can verify in a straightforward way that this direct sum yields the desired exact sequence.

Moreover, one computes
$$\dim {\rm Hom}(N_A,Q_{A^*})=n(n+1)/2+|\{((i,j),(k,l))\in A\times A^*\, :\, k<i\leq l<j\}|.$$
Again using the definition of $A^*$, the second summand is seen to equal zero.

Using the previous proposition, for every irreducible component of ${\rm Gr}_{\bf e}(M^2)$, we find a specific subrepresentation $U$ of $M^2$ (namely the one given by the above exact sequence) for which the tangent space $T_U{\rm Gr}_{\bf e}(M^2)\simeq{\rm Hom}(U,M^2/U)$ is of dimension $\dim{\rm Gr}_{\bf e}(M^2)$, proving generic reducedness.
\end{proof}

\subsection{Desingularizations of the irreducible components}
In this subsection we describe explicitly the desingularization of irreducible components of ${\rm Fl}^{{\bf r}^2}(V)$. 
In particular, we reprove that
the dimension of every component is equal to $n(n+1)/2$. Our main tool is the general construction of \cite{CFR2}.

Let $A$ be an arc diagram. The irreducible components are labeled by the non-crossing arc diagrams. 
For a non-crossing arc diagram $A$ the irreducible component is the closure $\overline{S_A}$ of the subset $S_A\subset {\rm Fl}^{{\bf r}^2}(V)$ 
of all tuples $(U_1,\ldots,U_n)$ such that
${\rm rank}((f_{j-1}\circ\ldots\circ f_i)|_{U_i}:U_i\rightarrow U_j)=r(A)_{i,j}$, where 
$$r(A)_{i,j}=i-|\{\mbox{arcs in $A$ starting in $[1,i]$ and ending in $[i+1,j]$}\}|.$$

The desingularization $R_A$ is formed by the collections of vector spaces $U_{i,j}\subset V$, $1\le i\le j\le n$ subject to the following conditions:
\begin{enumerate}
\item $U_{i,j}\subset {\rm Im} (f_{j-1}\circ\ldots\circ f_i),\ \dim U_{i,j}=r(A)_{i,j}$, \label{1}
\item $U_{i,j}\subset U_{i+1,j},\ f_j U_{i,j}\subset U_{i,j+1}$. \label{2}
\end{enumerate}
The map $R_A\to \overline{S_A}$ sends a collection $(U_{i,j})_{i,j}$ to $(U_{i,i})_{i=1}^n$. 

\begin{lemma}
Each variety $R_A$ is isomorphic to a tower $R_A=R_A(1)\to\dots\to R_A(N)={\rm pt}$, $N=n(n+1)/2$, where each map $R_A(k)\to R_A(k+1)$
is a fibration with the fibers being Grassmannians. 
\end{lemma}
\begin{proof}
A point in $R_A$ is a collection of spaces $U_{i,j}$. Our first step is to define the space $U_{1,n}$. This is a subspace of the one-dimensional space spanned by 
$f_{n-1}\circ\ldots\circ f_1$. Since $\dim U_{1,n}$ is either one or zero (depending on $r(A)_{1,n}$), we have no choice when fixing $U_{1,n}$.
We define $R_A(N)={\rm Gr}(r(A)_{1,n},1)$.     

In general, the space $R_A(k)$ is defined as the set of collections $(U_{i,j})_{i,j\in L(k)}$, where the cardinality 
of $L(k)$ is $N+1-k$ and 
$(i,j)\in L(k)$ implies $(i-1,j),(i,j+1)\in L(k)$, $1\le i\le j\le n$; 
the properties \eqref{1} and \eqref{2} are assumed to be fulfilled provided all the pairs $(i,j)$
popping up belong to $L(k)$. The sets $L(k)$ satisfy 
\[
\{(i,j), 1\le i\le j\le n\}=L(1)\supset L(2)\supset\dots\supset L(N)=\{(1,n)\}.
\]
Now assume that $(i-1,j), (i,j+1)\in L(k)$ and $(i,j)\notin L(k)$. Then we define $L(k-1)=L(k)\cup \{(i,j)\}$. Then there is a natural map
$R_A(k-1)\to R_A(k)$ and the fiber of such a map is parametrized by the subspaces $U_{i,j}$ such that $U_{i-1,j}\subset U_{i,j}$ and
$f_j U_{i,j}\subset U_{i,j+1}$. Hence such $U_{i,j}$ are parametrized by the Grassmannian 
\[
{\rm Gr} (r(A)_{i,j}-r(A)_{i-1,j}, f_{j}^{-1} U_{i,j+1}\cap {\rm Im} (f_{j-1}\circ\ldots\circ f_i)).   
\] 
\end{proof}

\begin{corollary}
${\rm Fl}^{{\bf r}^2}(V)$ is equidimensional of dimension $n(n+1)/2$.
\end{corollary}
\begin{proof}
One can show that for any non-crossing arc diagram $A$ the dimension of $R(A)$ is equal to
$\dim R_A=n(n+1)/2$. This implies the claim.
\end{proof}

\subsection{Cell decompositions}\label{Sec:CellDec}
We retain notations of previous sections. Thus $Q$ denotes an equioriented quiver of type $A_n$ (for some fixed integer $n\geq 1$) and let $M:=(V_\bullet, f_\bullet)$ be a $Q$-representation: 
$$
\xymatrix{
M:&V_1\ar^{f_1}[r]&V_2\ar^{f_2}[r]&\ar[r]\cdots\ar^{f_{n-1}}[r]&V_n
}
$$
We denote by $d_i:=\textrm{dim}_\mathbb{C}(V_i)$. Let $\mathbf{e}=(e_1,e_2,\cdots, e_n)\in \mathbf{Z}_{\geq0}^n$ be a dimension vector. We consider the corresponding quiver Grassmannian
$$
{\Gr}_\mathbf{e}(M):=\{(U_1,U_2,\cdots, U_n)\in \prod_{i=1}^n {\Gr}_{e_i}(V_i)|\,f_i(U_i)\subseteq U_{i+1}\}.
$$ 
In this subsection we show that all varieties ${\rm Gr}_\mathbf{e}(M)$ 
(for every $Q$--representation $M$ and dimension vector $\mathbf{e}$) admit  cellular decompositions
such that the points of each cell are all isomorphic as representations of $Q$.
\par
Let us introduce our candidates for the cells. It is well-known that  there exists a basis 
$$
\mathcal{B}_i=\{v_1^{(i)}, v_2^{(i)},\cdots, v_{d_i}^{(i)}\}
$$ 
of $V_i$ (for all $i=1,\cdots, n$) such that
\begin{equation}\label{Reduct1}
f_i(v_k^{(i)})\textrm{ is either zero or a basis vector }v_{k'}^{(i+1)}\in \mathcal{B}_{i+1}.
\end{equation}
We call $\mathcal{B}$ the \emph{standard basis} of $M$.   We renumber the basis vectors in such a way that if $f_i(v_{k}^{(i)})\neq 0$ then it equals the standard basis vector $v_{k}^{(i+1)}$ (with the same index $k$) where $k\geq 1$ is a positive integer.  We say that a maximal collection of vectors $\{v_{k}^{(i)}\}_i$ such that $f_i(v_{k}^{(i)})=v_{k}^{(i+1)}$ 
form the $k$--th \emph{segment} of $M$. 
\par
Our second reduction is the following: for every index $i$, we renumber the basis vectors of $\mathcal{B}_i$ so that
\begin{equation}\label{Reduct2}
v_k^{(i)}\in \ker f_i\qquad\Rightarrow\qquad v_j^{(i)}\in \ker f_i \qquad \forall j>k
\end{equation}
for every choice of $i$ and $k$. Such a renumbering is always possible (see Remark~\ref{Rem:Degree}). This property is equivalent to the following 
\begin{equation}\label{Reduct3}
f_i(v_k^{(i)}+\sum_{j>k}a_jv_j^{(i)})\neq 0\qquad\Rightarrow\qquad f_i(v_k^{(i)})\neq 0
\end{equation}
for every choice of $i$, $k$ and  of coefficients $a_j\in \mathbb{C}$.
\par
We can now construct our candidate for the cells (i.e. affine spaces) of ${\Gr}_\mathbf{e}(M)$. Following \cite{Cer}, we assign a degree to each standard basis vector as follows:
\begin{equation}\label{Eq:Degree}
\textrm{deg}(v_k^{(i)})=k.
\end{equation}
With this choice the vectors of every segment are homogenous of the same degree.  The 1--dimensional torus $T=\mathbb{C}^*$ acts on the quiver Grassmannian ${\rm Gr}_\mathbf{e}(M)$ as follows: given $\lambda\in T$ and a basis vector $v\in \bigcup_{i=1}^n\mathcal{B}_i$ we put:
$$
\lambda\cdot v=\lambda^{\textrm{deg}(v)} v
$$
and we extend this action to $M$ by linearity. It is easy to see that the map $v\mapsto \lambda.v$ is an automorphism of the $Q$--representation $M$: indeed, given a vector $v=\sum_k a_kv_k^{(i)}\in V_i$ and $\lambda\in T$, 
$$
f_i(\lambda\cdot v)=\sum_k a_k\lambda^k f_i(v_k^{(i)})=\sum_{k:f_i(v_k^{(i)})\neq0}a_k\lambda^k v_k^{(i+1)}=\lambda\cdot (f_i(v))
$$
Since the group $\textrm{Aut}_Q(M)$ naturally acts on ${\Gr}_\mathbf{e}(M)$, it follows that  $T\subset \textrm{Aut}(M)$ acts on ${\Gr}_\mathbf{e}(M)$. The set ${\Gr}_\mathbf{e}(M)^T$ of $T$--fixed points is finite and consists of all sub--representations of $M$ of dimension vector $\mathbf{e}$ which are spanned by standard basis vectors. Given $L\in {\Gr}_\mathbf{e}(M)^T$ we consider its attracting set
$$
\mathcal{C}(L):=\{N\in {\Gr}_\mathbf{e}(M)|\,\lim_{\lambda\rightarrow 0}\lambda\cdot N=L\}.
$$
\begin{theorem}\label{Thm:CellDecomp}
For every $L\in {\Gr}_\mathbf{e}(M)^T$, the subset $\mathcal{C}(L)\subseteq {\Gr}_\mathbf{e}(M)$ is an affine space and the quiver Grassmannian admits the cellular decomposition
\begin{equation}\label{Eq:CellDecomp}
{\Gr}_\mathbf{e}(M)=\coprod_{L\in {\Gr}_\mathbf{e}(M)^T}\mathcal{C}(L).
\end{equation}
\end{theorem}
\begin{proof}
The torus $T$ acts on each Grassmannian ${\Gr}_{e_i}(V_i)$ as \eqref{Eq:Degree} and induces a cell decomposition 
$$
{\Gr}_{e_i}(V_i)=\coprod_{L_i\in \Gr_{e_i}(V_i)^T}\mathcal{C}(L_i)\ \ 
\text{where}\ \ 
\mathcal{C}(L_i):=\{N_i\in {\rm Gr}_{e_i}(V_i)|\,\lim_{\lambda\rightarrow 0}\lambda\cdot N_i=L_i\}.
$$
Let $L\in {\Gr}_\mathbf{e}(M)^T$ and let us denote by $L_i$ the corresponding subspace of $V_i$ (for every vertex $i=1,\cdots, n$). Then, since the embedding ${\rm Gr}_\mathbf{e}(M)\subseteq \prod_{i=1}^n {\Gr}_{e_i}(V_i)$ is $T$--equivariant, 
\begin{equation}\label{Eq:Inters}
\mathcal{C}(L)={\rm Gr}_\mathbf{e}(M)\cap \prod_{i=1}^n \mathcal{C}(L_i).
\end{equation}
In order to finish the proof it remains to show that this intersection is an affine space.

It is easy to describe the affine space $\mathcal{C}(L_i)$: suppose that  $L_i$ is spanned by $\{v_{k_1}^{(i)},v_{k_2}^{(i)},\cdots, v_{k_{e_i}}^{(i)}\}$ for some set of  indices $K_i:=\{ k_1<k_2<\cdots<k_{e_i}\}$, then a point $N_i\in \mathcal{C}(L_i)$ is spanned by vectors $\{w_1^{(i)},\cdots, w_{e_i}^{(i)}\}$ of the form 
\begin{equation}\label{Eq:SpanningVectorsN}
w_s^{(i)}=v_{k_s}^{(i)}+\sum_{j>k_s,\, j\notin K_i} a_{j,s}^{(i)} v_j^{(i)}
\end{equation}
for some coefficients $a_{j,s}^{(i)}\in\mathbb{C}$. We claim that in the coordinates  $\{a_{j,s}^{(i)}\}$, the intersection \eqref{Eq:Inters} is described by the following equations:
\begin{equation}\label{EquationsCells}
a_{j,s}^{(i+1)}=a_{j,s}^{(i)}\quad\textrm{whenever }f_i(v_j^{(i)})\neq 0
\end{equation}
In particular, the claim shows that $\mathcal{C}(L)$ is a cell. The proof of the claim is straightforward: let us take a point $\{a_{j,s}^{(i)}\}\in\prod\mathcal{C}(L_i)$ which defines a collection $N=\{N_i\}_{i=1}^n$ of subspaces, each one spanned by vectors \eqref{Eq:SpanningVectorsN}. This point $N$ belongs to ${\Gr}_\mathbf{e}(M)$ if and only if $f_i(N_i)\subseteq N_{i+1}$ for $i=1,2,\cdots, n-1$. This means that $f_i(w_s^{(i)})$ must be in the span of  $\{w_{1}^{(i+1)},\cdots, w_{e_{i+1}}^{(i+1)}\}$. In view of \eqref{Reduct3}, if $f_i(w_s^{(i)})\neq 0$ then it equals 
$$
f_i(w_s^{(i)})=v_{k_s}^{(i+1)}+\sum_{j>k_s,\,j\notin K_i} a_{j,s}^{(i)}\, f_i(v_j^{(i)})=
v_{k_s}^{(i+1)}+\sum_{j>k_s,\,j\notin K_i,\, f_i(v_j^{(i)})\neq0} a_{j,s}^{(i)}\, v_j^{(i+1)}.
$$
(Notice that $v_{k_s}^{(i+1)}\in L_{i+1}$ since $L$ is a sub-representation of $M$.) This vector is in the span of $\{w_{1}^{(i+1)},\cdots, w_{e_{i+1}}^{(i+1)}\}$ if and only if it equals
$$
w_s^{(i+1)}=v_{k_s}^{(i+1)}+\sum_{j>k_s\,j\notin K_{i+1}} a_{j,s}^{(i+1)}\, v_j^{(i+1)}
$$
and this forces \eqref{EquationsCells}.
\end{proof}
\begin{remark}
A different numbering  than \eqref{Reduct2} does not produce cells. For example, let $V_1=\sspan\{v_1^{(1)}, v_2^{(1)}\}$, $V_2=\sspan\{v_2^{(2)}, v_3^{(2)}\}$ and $f_1(v_1^{(1)})=0$, $f_1(v_2^{(1)})=v_2^{(2)}$. 
\par
Then we see that condition \eqref{Reduct2} is not satisfied (since $f_1(v_1^{(1)})=0$ but $f_1(v_2^{(1)})\neq 0$). Let $\mathbf{e}=(1,1)$. The attracting set of the T--fixed point 
$L=\sspan\{v_1^{(1)}, v_2^{(2)}\}$ is given by
$$
\mathcal{C}(L)=\{(v_1^{(1)}+xv_2^{(1)},v_2^{(2)}+yv_3^{(2)})|\, xy=0\}
$$ 
which is not a cell. 
\end{remark}
Recall the stratification of ${\Gr}_\mathbf{e}(M)$ as union of locally closed subsets $\mathcal{S}(N)$ consisting of points $U\in {\Gr}_\mathbf{e}(M)$ isomorphic to $N$. 
\begin{corollary}\label{Cor:CellDecSN}
The cellular decomposition  \eqref{Eq:CellDecomp} induces a cellular decomposition 
\begin{equation}\label{Eq:CellDecompSN}
\mathcal{S}(N)=\coprod_{L\in {\Gr}_\mathbf{e}(M)^T:L\simeq N} \mathcal{C}(L).
\end{equation} 
In other words all the points of a cell $\mathcal{C}(L)$ are isomorphic to $L$ as $Q$--representations. 
\end{corollary}
\begin{proof}
We need to prove that each point $U$ of ${\Gr}_\mathbf{e}(M)$ is attracted by a torus fixed point which is isomorphic to it. By the explicit description of the cell $\mathcal{C}(L)$ given in the proof of the theorem, it follows that the ranks of the maps induced on each point of  $\mathcal{C}(L)$ are precisely the ranks of the same maps induced on $L$. Since isomorphism classes of $Q$--representations are parametrized by such ranks, this concludes the proof. 
\end{proof}
\begin{corollary}\label{Cor:SubTypeA}
The possible sub-representation types of M are given by torus fixed points. 
\end{corollary}
Notice that Corollary~\ref{Cor:SubTypeA} is not true for Dynkin quivers of type $D_n$ (see \cite[Example~4.3]{CFR1}).
\begin{remark}
It is worth noting that the dimension of the tangent space at ${\Gr}_\mathbf{e}(M)$ is not constant along each cell, in general (see Example~\ref{Example1}). So, it can happen that the center of a cell (i.e. its $T$--fixed point) is singular while the cell contains smooth points of ${\Gr}_\mathbf{e}(M)$.
\end{remark}
\begin{remark}
Equation~\eqref{EquationsCells} provides a  formula to compute the dimension of any given attracting cell $\mathcal{C}(L)$ for $L\in\Gr_\mathbf{e}(M)^T$.  
\end{remark}
\begin{remark}\label{Rem:Degree}
To each indecomposable Q--representation $U_{ij}$ ($1\leq i\leq j\leq n$) we assign the degree
\begin{equation}\label{Eq:DegreeIndec}
\deg U_{ij}:=j-i+1+{n+1\choose 2}-{j+1\choose 2}.
\end{equation}
It satisfies the following recursive relations: 
$$
\deg U_{ij}=\left\{
\begin{array}{lc}
\deg U_{i-1,j}-1&\textrm{ if }1<i\leq j\\
\deg U_{j-1,j-1}-1&\textrm{ if }i=1<j
\end{array}
\right.
$$
In particular, 
\begin{equation}\label{Eq:GradingProp}
\deg U_{ij}\geq\deg U_{rs}\Rightarrow j\leq s.
\end{equation}

This provides a total ordering on the set of indecomposable $Q$--representations.
\par
Given a $Q$--representation $M$, we order its indecomposable direct summands $M=\bigoplus_{i=1}^N M(i)$ so that
$$
i<j \Rightarrow \deg M(i)\leq \deg M(j).
$$ 
This ordering induces a  grading of the standard basis vectors of $M$ which satisfies \eqref{Reduct2}. Indeed, by assumption, the $k$--th segment is the span of $\{v_k^{(i)}\}_i$ and it is isomorphic to $M(k)$; every standard basis vector of such segment has degree $k$. If $v_k^{(i)}$ is defined and $f_i(v_k^{(i)})=0$, then the k--th segment of M is isomorphic to $U_{ri}$ for some $r\leq i$ and if $j>k$ then the $j$--th segment is $U_{st}$ with $\deg U_{ri}\leq \deg U_{st}$ and in view of \eqref{Eq:GradingProp}, $t\leq i$. In case $t=i$, this forces $f_i(v_j^{(i)})=0$ as desired. 
\end{remark}

We conclude this subsection with examples. They all concern the mf--linear degenerate flag variety ${\rm Fl}^{{\bf r}^2}(V)$, that we denote by $\mathcal{G}_n$ for simplicity. Recall that by definition ${\rm Fl}^{{\bf r}^2}(V)$ is the quiver Grassmannian ${\Gr}_\mathbf{e}(M^{(2)})$ where $\mathbf{e}=(1,2,\cdots, n)$ and  $M\simeq M^2$ is the representation of $Q$ given by
$M:=\bigoplus_{i=1}^n P_i\oplus\bigoplus_{j=1}^{n-1} I_j\oplus\bigoplus_{k=1}^nS_k.$
We order the indecomposable direct summands of $M$ as explained in Remark~\ref{Rem:Degree}.
Thus, the strings of $M$ (for $n=4$) are ordered as follows, from top to bottom:
$$
\xymatrix@R=0pt@C=12pt{
&&&\cdot&1\\
&&&\cdot&2\\
&&\cdot\ar[r]&\cdot&3\\
&\cdot\ar[r]&\cdot\ar[r]&\cdot&4\\
\cdot\ar[r]&\cdot\ar[r]&\cdot\ar[r]&\cdot&5\\
&&\cdot&&6\\
\cdot\ar[r]&\cdot\ar[r]&\cdot&&7\\
&\cdot&&&8\\
\cdot\ar[r]&\cdot&&&9\\
\cdot&&&&10\\
\cdot&&&&11}
$$

As shown in Theorem~\ref{t3} the irreducible components of $\mathcal{G}_n$ are labeled with non--crossing partitions on $n$ vertices. According to Corollary~\ref{Cor:CellDecSN}, each stratum $\mathcal{S}(N)$ of the mf-degenerate flag variety is divided into cells, parametrized by coordinate sub-representations of $M$ isomorphic to $N$. It is straightforward to compute the dimension of each cell in examples: consider the coefficient quiver of $M$, and arrange its strings as above. Given a $T$--fixed point $L$ of $\mathcal{S}(N)$, color black the vertices of the coefficient quiver of $M$ corresponding to its basis vectors and color white the remaining vertices. The dimension of the cell $\mathcal{C}(L)$ is given by counting the number of white vertices below each black \emph{source} (i.e. a source of the segments defining $L$). 

\begin{example}\label{Example1}
Let us consider the following $T$--fixed point $L$ of $\mathcal{G}_4$
$$
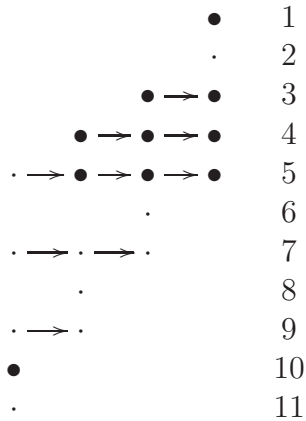
\xymatrix@R=0pt@C=12pt{
&&&\bullet&1\\
&&&\cdot&2\\
&&\bullet\ar[r]&\bullet&3\\
&\bullet\ar[r]&\bullet\ar[r]&\bullet&4\\
\cdot\ar[r]&\bullet\ar[r]&\bullet\ar[r]&\bullet&5\\
&&\cdot&&6\\
\cdot\ar[r]&\cdot\ar[r]&\cdot&&7\\
&\cdot&&&8\\
\cdot\ar[r]&\cdot&&&9\\
\bullet&&&&10\\
\cdot&&&&11}
$$
Its attracting cell $\mathcal{C}(L)$ has dimension 10 which is also the dimension of the whole variety $\mathcal{G}_4$. In view of Corollary~\ref{Cor:CellDecSN},  it follows that the  stratum $\mathcal{S}(L)$ is generic (its closure is an irreducible component of $\mathcal{G}_4$) and indeed it is indexed by the non--crossing partition $\{(1,2)\}$. This cell is interesting, because its center is non--smooth (the tangent space of $\mathcal{G}_4$ at $L$ has dimension 11) but the cell contains smooth points (since $\mathcal{G}_n$ is generically reduced for every $n$). 
\end{example}
\subsection{Normal flat locus}
We have shown above that the flat locus consists of those $Q$--representations $M$ such that $M\leq_{\textrm{deg}} M^2$. Inside the flat locus, the irreducible flat locus consists of  those $M$ such that $M\Leq M^1$. Theorem~\ref{t5} shows that $M$ lies in the flat locus but not in the flat irreducible locus if and only if there exists an index $i$ such that $M(\mathbf{a^i})\Leq M\Leq M^2$. 
Figure~\ref{Fig} summarizes the situation.
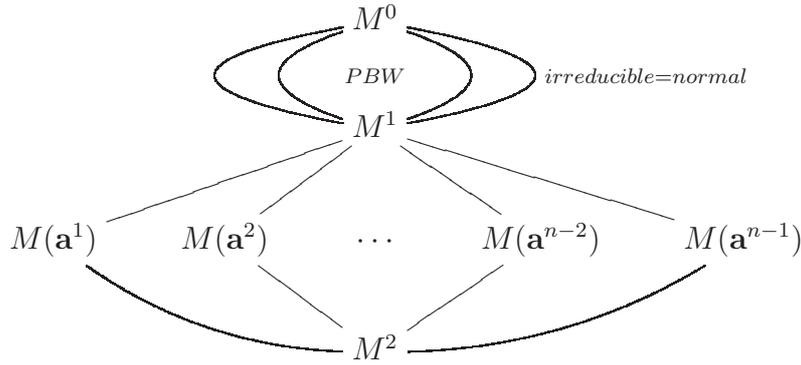
\begin{figure}\
$$
\xymatrix{
&&M^0&&\\
&&M^1\ar@/^5pc/@{-}[u]\ar@/^3pc/@{-}[u]\ar@/_5pc/@{-}_{irreducible=normal}[u]\ar@/_3pc/@{-}[u]\ar|{PBW}@{}[u]&&\\
M(\mathbf{a}^1)\ar@{-}[urr]&M(\mathbf{a}^2)\ar@{-}[ur]&\cdots&M(\mathbf{a}^{n-2})\ar@{-}[ul]&M(\mathbf{a}^{n-1})\ar@{-}[ull]\\
&&M^2\ar@/^1pc/@{-}[ull]\ar@/_1pc/@{-}[urr]\ar@{-}[ul]\ar@{-}[ur]&&
}
$$
\caption{Flat locus}
\label{Fig}
\end{figure}
The next result shows that the fiber over a point lying in the flat locus is normal if and only if  the point lies in the irreducible flat locus. In other words a flat linear degeneration of the complete flag variety is normal if and only if it is irreducible. 

\begin{theorem}\label{normallocus}
For every $i$, the quiver Grassmannian $X:={\Gr}_{(1,2,\cdots, n)}(M(\mathbf{a}^i))$  has singularities in codimension~1 and hence it is not normal. A flat linear degeneration of the complete flag variety is normal if and only if it is irreducible. 
\end{theorem}
\begin{proof}
Let $\mathbf{a}^i=(0,\cdots, 0, 2,0,\cdots,0)$, with $2$ in the $i$--th position. The corresponding representation  $M(\mathbf{a}^i)$ is the $Q$--representation
$$
\xymatrix@R=5pt{
M(\mathbf{a}^i):=&\mathbb{C}^{n+1}\ar^1[r]&\cdots\ar^1[r]&\mathbb{C}^{n+1}\ar^{\textrm{pr}_{i,i+1}}[r]&\mathbb{C}^{n+1}\ar^1[r]&\cdots\ar^1[r]&\mathbb{C}^{n+1}
}
$$
where the map $\textrm{pr}_{i,i+1}$ is between vertex $i$ and vertex $i+1$. Its decomposition is $M(\mathbf{a}^i):=P_1^{\oplus(n+1-2)}\oplus P_{i+1}^{\oplus 2}\oplus I_{i}^{\oplus 2}$. Since $M(\mathbf{a}^i)$ degenerates to $M^{2}$, we know that $X$ is Cohen--Macaulay. Let us show that it has singularities in codimension~1. Recall that the irreducible components of $X$ have all the same dimension $d:=n(n+1)/2$ and they are closures of attracting sets of some torus fixed points. We show that there exist two irreducible components $I_1=\overline{\mathcal{C}(P)}$ and $I_2=\overline{\mathcal{C}(R)}$ (for suitable subrepresentations $P$ and $R$ of $M(\mathbf{a}^2)$) such that the intersection $I_1\cap I_2$ contains a cell $\mathcal{C}(L)$ of dimension $d-1$. 

We order the indecomposable direct summands $M_1,M_2,\cdots, M_{n+3}$ of $M(\mathbf{a}^i)$ as $M_1=M_2=P_{i+1}$, $M_3=\cdots=M_{n+1}=P_1$, $M_{n+2}=M_{n+3}=I_{i}$. This ordering induces an ordering of the standard basis of $M(\mathbf{a}^i)$ (see Section~\ref{Sec:CellDec}):  $\{v_{k}^{(\ell)}|\, k\in[1,n+3],\, \ell\textrm{ is in the support of }M_k\}$. The one--dimensional torus $T\simeq \mathbb{C}^\ast$ embeds into $\textrm{Aut}_Q(M(\mathbf{a}^i))$ by rescaling this basis elements: $\lambda\cdot v_k^{(\ell)}:=\lambda^kv_k^{(\ell)}$. This induces an action of $T$ on $X$. The $T$--fixed points are precisely the coordinate subrepresentations of $M(\mathbf{a}^i)$ of the prescribed dimension vector. Moreover, as shown in Section~\ref{Sec:CellDec} the attracting sets of $T$--fixed points are cells. 
\begin{example}\label{Ex:Ma2}
For n=4, $M(\mathbf{a}^2)$ is given by
$$
\xymatrix@R=0pt@C=12pt{
&&v_1^{(3)}\ar[r]&v_1^{(4)}&1\\
&&v_2^{(3)}\ar[r]&v_2^{(4)}&2\\
v_3^{(1)}\ar[r]&v_3^{(2)}\ar[r]&v_3^{(3)}\ar[r]&v_3^{(4)}&3\\
v_4^{(1)}\ar[r]&v_4^{(2)}\ar[r]&v_4^{(3)}\ar[r]&v_4^{(4)}&4\\
v_5^{(1)}\ar[r]&v_5^{(2)}\ar[r]&v_5^{(3)}\ar[r]&v_5^{(4)}&5\\
v_6^{(1)}\ar[r]&v_6^{(2)}&&&6\\
v_7^{(1)}\ar[r]&v_7^{(2)}&&&7
}
$$
In this figure, the indecomposable direct summands of $M(\mathbf{a}^2)$ are displayed as segments (see Section~\ref{Sec:CellDec}) and they are numbered from top to bottom according to the enumeration shown on the right. 
\end{example}

We are now ready to define the torus fixed points $P$ and $R$ of $X$. Recall that a torus fixed point is a coordinate subrepresentations of $M(\mathbf{a}^i)$ in the basis $\{v_k^{(\ell)}\}$. Such a representation is given by a collection of sub--segments of the segments forming $M(\mathbf{a}^i)$ and it is uniquely determined by its generators, i.e. the sources of such sub--segments. Let $P\in X$ be the sub-representation of $M(\mathbf{a}^i)$ generated by 
$$
P=\langle v_3^{(1)}, v_4^{(2)},\cdots, v_{i+2}^{(i)}, v_1^{(i+1)}, v_2^{(i+2)}, v_{i+3}^{(i+3)}, \cdots, v_{n}^{(n)}\rangle_{\mathbb{C}Q}
$$
We define $R$ to be the sub-representation generated by
$$
R=\langle v_3^{(1)}, v_4^{(2)},\cdots, v_{i+1}^{(i-1)}, v_{n+2}^{(i)}, v_1^{(i+1)}, v_2^{(i+1)}, v_{i+2}^{(i+2)}\cdots, v_{n}^{(n)} \rangle_{\mathbb{C}Q}
$$
\begin{example}\label{Ex:Irr}
In our running example~\ref{Ex:Ma2} (n=4, i=2), the T--fixed points $P$ and $R$ are given by
$$
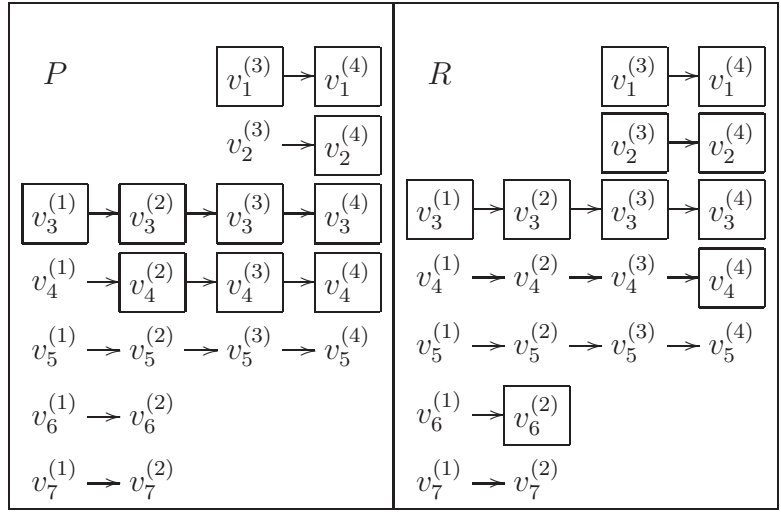
\begin{tabular}{|l|r|}
\hline
\xymatrix@R=3pt@C=12pt{
&&&\\
P&&*+[F]{v_1^{(3)}}\ar[r]&*+[F]{v_1^{(4)}}\\
&&v_2^{(3)}\ar[r]&*+[F]{v_2^{(4)}}\\
*+[F]{v_3^{(1)}}\ar[r]&*+[F]{v_3^{(2)}}\ar[r]&*+[F]{v_3^{(3)}}\ar[r]&*+[F]{v_3^{(4)}}\\
v_4^{(1)}\ar[r]&*+[F]{v_4^{(2)}}\ar[r]&*+[F]{v_4^{(3)}}\ar[r]&*+[F]{v_4^{(4)}}\\
v_5^{(1)}\ar[r]&v_5^{(2)}\ar[r]&v_5^{(3)}\ar[r]&v_5^{(4)}\\
v_6^{(1)}\ar[r]&v_6^{(2)}&&\\
v_7^{(1)}\ar[r]&v_7^{(2)}&&
}
&
\xymatrix@R=3pt@C=12pt{
&&&\\
R&&*+[F]{v_1^{(3)}}\ar[r]&*+[F]{v_1^{(4)}}\\
&&*+[F]{v_2^{(3)}}\ar[r]&*+[F]{v_2^{(4)}}\\
*+[F]{v_3^{(1)}}\ar[r]&*+[F]{v_3^{(2)}}\ar[r]&*+[F]{v_3^{(3)}}\ar[r]&*+[F]{v_3^{(4)}}\\
v_4^{(1)}\ar[r]&v_4^{(2)}\ar[r]&v_4^{(3)}\ar[r]&*+[F]{v_4^{(4)}}\\
v_5^{(1)}\ar[r]&v_5^{(2)}\ar[r]&v_5^{(3)}\ar[r]&v_5^{(4)}\\
v_6^{(1)}\ar[r]&*+[F]{v_6^{(2)}}&&\\
v_7^{(1)}\ar[r]&v_7^{(2)}&&
}
\\\hline
\end{tabular}
$$
The dimension of $\mathcal{C}(P)$ is computed by counting  for each generator of $P$, the number of vertices which lie below it and which are not in P: this is $4+3+2+1=10$. Similarly the dimension of $\mathcal{C}(R)$ is given by $4+1+2+2+1=10$. 
\end{example}
Formula \eqref{EquationsCells} implies that $\textrm{dim}\,\mathcal{C}(P)=\textrm{dim}\,\mathcal{C}(R)=d$ and hence both $\mathcal{I}_1:=\overline{\mathcal{C}(P)}$ and $\mathcal{I}_1:=\overline{\mathcal{C}(R)}$ are irreducible components of $X$. We consider the subrepresentation $Q$ generated by
$$
Q=\langle v_3^{(1)}, v_4^{(2)},\cdots, v_{i+1}^{(i-1)}, v_{n+2}^{(i)}, v_1^{(i+1)}, v_{i+2}^{(i+1)}, v_2^{(i+2)}, v_{i+3}^{(i+3)}\cdots, v_{n}^{(n)} \rangle_{\mathbb{C} Q}.
$$
\begin{example}\label{Ex:Q}
In our running example~\ref{Ex:Ma2}, the following is the subrepresentation $Q$:
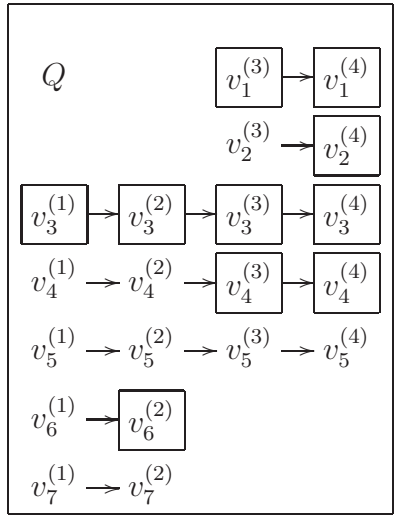
$$
\begin{tabular}{|c|}
\hline
\xymatrix@R=3pt@C=12pt{
&&&\\
Q&&*+[F]{v_1^{(3)}}\ar[r]&*+[F]{v_1^{(4)}}\\
&&v_2^{(3)}\ar[r]&*+[F]{v_2^{(4)}}\\
*+[F]{v_3^{(1)}}\ar[r]&*+[F]{v_3^{(2)}}\ar[r]&*+[F]{v_3^{(3)}}\ar[r]&*+[F]{v_3^{(4)}}\\
v_4^{(1)}\ar[r]&v_4^{(2)}\ar[r]&*+[F]{v_4^{(3)}}\ar[r]&*+[F]{v_4^{(4)}}\\
v_5^{(1)}\ar[r]&v_5^{(2)}\ar[r]&v_5^{(3)}\ar[r]&v_5^{(4)}\\
v_6^{(1)}\ar[r]&*+[F]{v_6^{(2)}}&&\\
v_7^{(1)}\ar[r]&v_7^{(2)}&&
}\\\hline
\end{tabular}
$$
\end{example}
We notice that $Q$ is obtained from $P$ by  replacing  $v_{i+2}^{(i)}\mapsto v_{n+2}^{(i)}$ and by keeping all the other basis elements. Geometrically, this map represents  a ``positive'' $T$--fixed vector of the tangent space of $X$ at $P$, which is the direction of a 1--dimensional $T$--fixed subvariety of $X$ whose limit points are precisely $P$ and $Q$. Notice that by the tangent space formula, $T_P(X)\simeq \Hom_Q(P, M(\mathbf{a}^1)/P)$ this vector corresponds to the (unique up to scalars) non--zero homomorphism from $P_i$ to $I_i$.   Similarly, $Q$  is obtained from $R$ by $v_{2}^{(i+1)}\mapsto v_{i+2}^{(i+1)}$ which has the same geometric interpretation (this corresponds to a non--zero homomorphism from $P_{i+1}$ to $I_{i+1}$). In particular, $Q$ and all its attracting cell lies in $\mathcal{I}_1\cap \mathcal{I}_2$. 

It remains to show that $\textrm{dim }\mathcal{C}(Q)=d-1$. From the dimension formula \eqref{EquationsCells} we immediately get $\textrm{dim }\mathcal{C}(Q)=\textrm{dim }\mathcal{C}(P)-(n+2-i)+(n+1-i)=\textrm{dim }\mathcal{C}(P)-1$.

The rest of the proof follows from the fact that normality is preserved under deformations. In particular, the fiber over a point $M$ such that $M(\mathbf{a}^i)\Leq M$ is not normal. On the other hand, if $M$ lies in the irreducible flat locus, then it degenerates to $M^1$. Since the degenerate flag variety (which is the fiber over $M^1$) is normal, it follows that the fiber over $M$ is normal as well. 
\end{proof}

\subsection{Geometry of linear degenerations - the flat irreducible locus}\label{flatirr}

Since the degenerate flag variety ${\rm Fl}^{{\bf r}^1}(V)$ is the special fiber of $\pi:\pi^{-1}(U_{\rm flat,irr})\rightarrow U_{\rm flat,irr}$, we can conclude from Theorem \ref{ag}:

\begin{theorem} All linear degenerations of flag varieties over $U_{\rm flat, irr}$ are reduced irreducible normal local complete intersection varieties.
\end{theorem}

Moreover, we can alternatively characterize the irreducible flat locus as the open subset of the flat locus where the fibers are normal varieties; see Theorem \ref{normallocus}. We can also characterize the PBW locus inside the flat irreducible locus.

\begin{theorem} Inside $U_{\rm flat,irr}$, the locus $U_{\rm PBW}$ consists of those points whose fibers are Schubert quiver Grassmannians (see Remark~\ref{Rem:Schubert}).
\end{theorem}

\begin{proof} By the main result of \cite{CFRSchubert}, the image of the natural embedding of a quiver Grassmannian ${\rm Gr}_{\bf e}(M)$ into a flag manifold is a union of Schubert varieties if and only if $M$ is a catenoid (see \cite[Definition~1.1]{CFRSchubert} for the definition). By the explicit description of orbits in $U_{\rm flat,irr}$ in Proposition \ref{proprhyme}, we see that this holds true if and only if $M$ belongs to $U_{\rm PBW}$.
\end{proof}

Conjecturally, there is an alternative characterization of the PBW locus inside the flat irreducible locus:

\begin{conjecture} There exists a flat solvable group scheme $\Gamma\rightarrow U_{\rm flat,irr}$ acting on $\pi^{-1}(U_{\rm flat,irr})\rightarrow U_{\rm flat,irr}$ with a dense orbit.
\end{conjecture}

As an example, we consider the case $n=3$. We can work over the transversal slice $T$, which consists of all pairs $(f_1,f_2)$ of the form
$$f_1=\left(\begin{array}{cccc}1&0&0&0\\ 0&\lambda_{11}&\lambda_{12}&0\\ 0&0&1&0\\ 0&0&0&1\end{array}\right),\; f_2=\left(\begin{array}{cccc}1&0&0&0\\ 0&1&\lambda_{12}&0\\ 0&0&\lambda_{22}&0\\ 0&0&0&1\end{array}\right).$$
It consists of five orbits, given by the following equations:
\begin{enumerate}
\item $\lambda_{11},\lambda_{22}\not=0$,
\item $\lambda_{11}\not=0,\lambda_{22}=0$,
\item $\lambda_{11}=0,\lambda_{22}\not=0$,
\item $\lambda_{11}=\lambda_{22}=0,\lambda_{12}\not=0$,
\item $\lambda_{11}=\lambda_{22}=\lambda_{12}=0$.
\end{enumerate}
The PBW consists of all orbits except the fourth one.

The following solvable group scheme verifies the conjecture in this case. It is given by triples $(g_1(\lambda),g_2(\lambda),g_3(\lambda))$:
$$\left(\begin{array}{cccc}1&0&0&0\\ x_{21}&1&0&0\\ x_{31}&\lambda_{11}x_{32}&1+2\lambda_{12}x_{32}&0\\ x_{41}&\lambda_{11}x_{42}&\lambda_{22}x_{43}+2\lambda_{12}x_{42}&1\end{array}\right),$$
$$\left(\begin{array}{cccc}1&0&0&0\\ \lambda_{11}x_{21}+\lambda_{12}x_{31}&1+\lambda_{12}x_{32}&\lambda_{12}^2x_{32}&0\\
x_{31}&x_{32}&1+\lambda_{12}x_{32}&0\\ x_{41}&x_{42}&\lambda_{22}x_{43}+\lambda_{12}x_{42}&1\end{array}\right),$$
$$\left(\begin{array}{cccc}1&0&0&0\\ \lambda_{11}x_{21}+2\lambda_{12}x_{31}&1+2\lambda_{12}x_{32}&0&0\\
\lambda_{22}x_{31}&\lambda_{22}x_{32}&1&0\\ x_{41}&x_{42}&x_{43}&1\end{array}\right)$$

\end{document}